\documentclass{article}
\usepackage{amsmath}
\usepackage{amssymb}
\usepackage{xcolor}
\usepackage{nccmath}
\usepackage{graphicx}

\newtheorem{theorem}{Theorem}

\newtheorem{corollary}[theorem]{Corollary}

\newtheorem{lemma}[theorem]{Lemma}

\newtheorem{proposition}[theorem]{Proposition}
\newtheorem{remark}[theorem]{Remark}

\DeclareMathOperator{\dom}{dom}

\DeclareMathOperator{\range}{range}
\DeclareMathOperator{\gph}{gph}

\DeclareMathOperator{\cl}{cl}

\newenvironment{proof}[1][Proof]{\noindent\textbf{#1.} }{\ \rule{0.5em}{0.5em}}
\input{tcilatex}
\begin{document}

\title{Feasibility problems via paramonotone operators in a convex setting%
\thanks{%
This research has been partially supported by Grant PGC2018-097960-B-C21
from MICINN, Spain, and ERDF, "A way to make Europe", European Union, and
Grant PROMETEO/2021/063 from Generalitat Valenciana, Spain. The third author
has also been partially supported by the Severo Ochoa Programme for Centres
of Excellence in R\&D [SEV-2015-0563]. He is affiliated with MOVE (Markets,
Organizations and Votes in Economics).}}
\author{J. Camacho \thanks{
Center of Operations Research, Miguel Hern\'{a}ndez University of Elche,
03202 Elche (Alicante), Spain (j.camacho@umh.es, canovas@umh.es,
parra@umh.es).} \and M.J. C\'{a}novas\footnotemark[2] \and J.E. Mart\'{\i}%
nez-Legaz \thanks{
Departament d'Economia i d'Hist\`{o}ria Econ\`{o}mica, Universitat Aut\`{o}%
noma de Barcelona, and BGSMath, Barcelona, Spain
(JuanEnrique.Martinez.Legaz@uab.cat).} \and J. Parra\footnotemark[2]}
\date{}
\maketitle

\begin{abstract}
This paper is focused on some properties of paramonotone operators on Banach
spaces and their application to certain feasibility problems for\textbf{\ }%
convex sets in a Hilbert space and convex systems in the Euclidean space. In
particular, it shows that operators that are simultaneously paramonotone and
bimonotone are constant on their domains, and this fact is applied to tackle
two particular situations. The first one, closely related to simultaneous
projections, deals with a finite amount of convex sets with an empty
intersection and tackles the problem of finding the smallest perturbations
(in the sense of translations) of these sets to reach a nonempty
intersection. The second is focused on the distance to feasibility;
specifically, given an inconsistent convex inequality system, our goal is to
compute/estimate the smallest right-hand side perturbations that reach
feasibility. We advance that this work derives lower and upper estimates of
such a distance, which become the exact value when confined to linear
systems.

\textbf{Key words.} Distance function, convex inequalities, distance to
feasibility, paramonotone operators, displacement mapping\newline

\bigskip

\noindent \textbf{Mathematics Subject Classification: }47N10, 47H05, 52A20,
90C31, 49K40
\end{abstract}

\section{Introduction}

The present paper is focused on paramonotone operators with applications to
certain feasibility problems for convex sets in a Hilbert space and convex
inequality systems in $\mathbb{R}^{n}.$ To start with, we recall some basic
properties of operators in Banach spaces. Let $X$ be a real Banach space,
with topological dual $X^{\ast },$ and denote by $\left\langle \cdot ,\cdot
\right\rangle $ the corresponding canonical pairing. A set-valued operator $%
T:X\rightrightarrows X^{\ast }$ is said to be \emph{monotone} if 
\begin{equation*}
\left\langle x-y,x^{\ast }-y^{\ast }\right\rangle \geq 0\text{ whenever }%
\left( x,x^{\ast }\right) ,\left( y,y^{\ast }\right) \in \mathrm{gph}T,
\end{equation*}%
where $\mathrm{gph}T:=\{\left( x,x^{\ast }\right) \in X\times X^{\ast
}:x^{\ast }\in T\left( x\right) \}$ is the graph of $T.$ In the case when
both $T$ and $-T$ are monotone, then $T$ is called \emph{bimonotone.} If $T$
is monotone and, in addition, $\mathrm{gph}T$ is maximal in the sense of
inclusion order, it is said to be \emph{maximally monotone. }A well-known
example of maximally monotone operator is the subdifferential operator of a
proper, lower semicontinuous (lsc, for short), convex function $%
f:X\rightarrow \left] -\infty ,+\infty \right] ,$ denoted by $\partial f$
(see Section 2 for details). Monotone operators are fundamental tools of
nonlinear analysis and optimization; see, e.g., the books \cite%
{BauschkeCombettes11, BorweinVanderwerff10, BI08, Phelps, RocKWet09,
Simons08, Zalinescu02}. A monotone operator $T$ is called \emph{paramonotone}
if the following implication holds:%
\begin{equation*}
\left. 
\begin{array}{c}
\left( x,x^{\ast }\right) ,\left( y,y^{\ast }\right) \in \mathrm{gph}T \\ 
\left\langle x-y,x^{\ast }-y^{\ast }\right\rangle =0%
\end{array}%
\right\} \Rightarrow \left( x,y^{\ast }\right) ,\left( y,x^{\ast }\right)
\in \mathrm{gph}T.
\end{equation*}%
The term paramonotonicity was introduced in \cite{CensorIusemZenios98}
(although the condition was previously presented in \cite{Bruck75} without a
name). The initial motivation for the introduction of paramonotone operators
comes from its crucial role regarding interior point methods for variational
inequalities (see again \cite{Bruck75} and \cite{CensorIusemZenios98}, and
also \cite{Iusem98}). Some important examples of paramonotone operators are
gathered in Section 2. At this moment, let us mention that subdifferentials
of proper lsc convex functions enjoy this property (see \cite[Proposition 2.2%
]{Iusem98} in the Euclidean space and \cite[Fact 3.1]{BausWangYao10} for its
extension to Banach spaces).

Looking at the applications of Sections 4 and 5, we are interested in
operators of the form $T_{1}\cap \left( -T_{2}\right) ,$ where $T_{1},$ $%
T_{2}:X\rightrightarrows X^{\ast }$ are paramonotone, which are also
paramonotone and, additionally, bimonotone; this fact entails that $%
T_{1}\cap \left( -T_{2}\right) $ is constant on its domain (as shown in
Corollary \ref{cor bimon param implies constant}); recall that the domain of
an operator $T$ is given by $\mathrm{dom}T:=\left\{ x\in X\mid T\left(
x\right) \neq \emptyset \right\} .$ Observe that 
\begin{equation}
\mathrm{dom}\left( T_{1}\cap \left( -T_{2}\right) \right) =\left\{ x\in
X\mid 0\in \left( T_{1}+T_{2}\right) \left( x\right) \right\} ,
\label{domT1y(-T2)}
\end{equation}%
which, in the particular case $T_{i}=\partial f_{i},$ $i=1,2,$ where the $%
f_{i}$'s are proper, lsc and convex, is known to coincide with 
\begin{equation}
\arg \min \left( f_{1}+f_{2}\right) ,  \label{eq_argmin_f1+f2}
\end{equation}%
i.e., with the set of (global) minima of $f_{1}+f_{2},$ provided that a
regularity condition ensuring $\partial f_{1}+\partial f_{2}=\partial \left(
f_{1}+f_{2}\right) $ is satisfied. These comments easily generalize to the
sum of a finite number of functions (see Section 3 for details) and are
applied to particular problems of the form%
\begin{equation}
\begin{array}{ll}
\underset{x\in X}{\text{minimize}}\text{ } & \sum\limits_{i=1}^{m}f_{i}%
\left( x\right) ,%
\end{array}
\label{eq_sum_fi}
\end{equation}%
where all $f_{i}$'s are proper lsc convex functions on $X.$

Now we present two applications discussed in the paper. The first one is
developed in a Hilbert space $X$ whose norm, associated with the
corresponding inner product, is denoted by $\left\Vert \cdot \right\Vert $.
It deals with a finite number of nonempty closed convex sets $%
S_{1},S_{2},...,S_{m}$ such that $\cap _{i=1}^{m}S_{i}=\emptyset $ and is
focused on the optimization problem given by 
\begin{equation}
\begin{array}{ll}
\underset{x\in X}{\text{minimize}}\text{ } & \sum\limits_{i=1}^{m}\alpha
_{i}d\left( x,S_{i}\right) ^{p},%
\end{array}
\label{eq_min_dist_k_weighted}
\end{equation}%
where $\alpha _{i}>0,$ $i=1,...,m;$ without loss of generality we assume $%
\sum\limits_{i=1}^{m}\alpha _{i}=1,$ $p\geq 1$ and $d\left( x,S_{i}\right) $
denotes the distance from point $x$ to set $S_{i},$ $i=1,...,m.$ The
following proposition establishes that (\ref{eq_min_dist_k_weighted}) is
equivalent to the problem:%
\begin{equation}
\begin{array}{ll}
\text{minimize } & \left\Vert u\right\Vert _{\alpha ,p} \\ 
\text{ subject to } & \cap _{i=1}^{m}\left( S_{i}+u_{i}\right) \neq
\emptyset , \\ 
& u=\left( u_{1},...,u_{m}\right) \in X^{m},%
\end{array}
\label{eq_problem_u_in_Xm}
\end{equation}%
where $\left\Vert u\right\Vert _{\alpha ,p}$ denotes the weighted $p$-norm
in space $X^{m}$ defined as 
\begin{equation}
\left\Vert u\right\Vert _{\alpha ,p}=\left( \sum\limits_{i=1}^{m}\alpha
_{i}\left\Vert u_{i}\right\Vert ^{p}\right) ^{1/p}.  \label{eq_norm in Xm}
\end{equation}%
This equivalence was already observed in \cite[Section 4]{BauCoLu04} for
Euclidean spaces;\ we include a proof for the sake of completeness.

\begin{proposition}
\label{equiv}A point $\overline{u}=\left( \overline{u}_{1},...,\overline{u}%
_{m}\right) \in X^{m}$ is an optimal solution to (\ref{eq_problem_u_in_Xm})
if and only if there exist an optimal solution $\overline{x}$ to (\ref%
{eq_min_dist_k_weighted}) such that $\overline{u}_{i}=\overline{x}%
-P_{i}\left( \overline{x}\right) ,$ $i=1,...,m$, with $P_{i}\left( \overline{%
x}\right) $ being the best approximation of $\overline{x}$ in $S_{i}.$
\end{proposition}

\begin{proof}
Let $\overline{u}=\left( \overline{u}_{1},...,\overline{u}_{m}\right) \in
X^{m}$ be an optimal solution to (\ref{eq_problem_u_in_Xm}), and take $%
\overline{x}\in \cap _{i=1}^{m}\left( S_{i}+\overline{u}_{i}\right) .$ There
exist $s_{i}\in S_{i},$ $i=1,...,m,$ such that $\overline{x}=s_{i}+\overline{%
u}_{i}.$ For every $x\in X,$ we have%
\begin{eqnarray*}
\sum\limits_{i=1}^{m}\alpha _{i}d\left( \overline{x},S_{i}\right) ^{p} &\leq
&\sum\limits_{i=1}^{m}\alpha _{i}\left\Vert \overline{x}-s_{i}\right\Vert
^{p}=\sum\limits_{i=1}^{m}\alpha _{i}\left\Vert \overline{u}_{i}\right\Vert
^{p}\leq \sum\limits_{i=1}^{m}\alpha _{i}\left\Vert x-P_{i}\left( x\right)
\right\Vert ^{p} \\
&=&\sum\limits_{i=1}^{m}\alpha _{i}d\left( x,S_{i}\right) ^{p}.
\end{eqnarray*}%
To justify the latter inequality, observe that $\cap _{i=1}^{m}\left(
S_{i}+x-P_{i}\left( x\right) \right) \neq \emptyset ,$ because from the
equalities $x=P_{i}\left( x\right) +x-P_{i}\left( x\right) ,$ $i=1,...,m,$
it immediately follows that $x\in \cap _{i=1}^{m}\left( S_{i}+x-P_{i}\left(
x\right) \right) .$ Thus, $\overline{x}$ is an optimal solution to (\ref%
{eq_min_dist_k_weighted}). Furthermore, setting $x=\overline{x},$ we also
deduce that $d\left( \overline{x},S_{i}\right) =\left\Vert \overline{x}%
-s_{i}\right\Vert ,$ $i=1,...,m$, that is, $s_{i}=P_{i}\left( \overline{x}%
\right) ,$ so that $\overline{u}_{i}=\overline{x}_{i}-s_{i}=\overline{x}%
_{i}-P_{i}\left( \overline{x}\right) .$

Conversely, let $\overline{x}$ be an optimal solution to (\ref%
{eq_min_dist_k_weighted}), $u$ be a feasible solution to (\ref%
{eq_problem_u_in_Xm}), and take $x\in \cap _{i=1}^{m}\left(
S_{i}+u_{i}\right) .$ Then, there exist $s_{i}\in S_{i},$ $i=1,...,m,$ such
that $x=s_{i}+u_{i},$ and we have%
\begin{eqnarray*}
\left\Vert u\right\Vert _{\alpha ,p}^{p} &=&\sum\limits_{i=1}^{m}\alpha
_{i}\left\Vert u_{i}\right\Vert ^{p}=\sum\limits_{i=1}^{m}\alpha
_{i}\left\Vert x-s_{i}\right\Vert ^{p}\geq \sum\limits_{i=1}^{m}\alpha
_{i}d\left( x,S_{i}\right) ^{p}\geq \sum\limits_{i=1}^{m}\alpha _{i}d\left( 
\overline{x},S_{i}\right) ^{p} \\
&=&\sum\limits_{i=1}^{m}\alpha _{i}\left\Vert \overline{x}-P_{i}\left( 
\overline{x}\right) \right\Vert ^{p},
\end{eqnarray*}%
which shows that the point $\overline{u}=\left( \overline{u}_{1},...,%
\overline{u}_{m}\right) ,$ with $\overline{u}_{i}:=\overline{x}%
_{i}-P_{i}\left( \overline{x}\right) ,$ $i=1,...,m$, is an optimal solution
to (\ref{eq_problem_u_in_Xm}).\bigskip
\end{proof}

According to Proposition \ref{equiv}, problem (\ref{eq_min_dist_k_weighted})
is equivalent to that of finding the smallest translations of the sets $%
S_{i} $ that achieve a nonempty intersection.

The second application, developed in Section 5, deals with convex inequality
systems in\textbf{\ }$\mathbb{R}^{n}$ parameterized with respect to the
right-hand side (RHS, in brief)$,$ 
\begin{equation}
\sigma \left( b\right) :=\left\{ g_{i}(x)\leq b_{i},\;i=1,\ldots ,m\right\} ,
\label{eq_sigma_b}
\end{equation}%
where $x\in \mathbb{R}^{n}$ is the vector of decision variables and, for
each $i\in 1,\ldots ,m,$ $g_{i}:\mathbb{R}^{n}\rightarrow \mathbb{R}$ is a
(finite-valued) convex function on $\mathbb{R}^{n},$ and $%
(b_{i})_{i=1,\ldots ,m}\equiv b\in \mathbb{R}^{m}$. Taking a nominal $%
\overline{b}\in \mathbb{R}^{m}$ such that $\sigma \left( \overline{b}\right) 
$ is inconsistent (i.e., there is no $x\in \mathbb{R}^{n}$ satisfying all
inequalities of $\sigma \left( \overline{b}\right) ),$ our aim is to
estimate the distance in $\mathbb{R}^{m}$ endowed with any $p$-norm, with $%
p\geq 2,$ from $\overline{b}$ to the set of parameters $b$ such that $\sigma
\left( b\right) $ is consistent. This \emph{distance to feasibility }can be
computed by solving the following problem:%
\begin{equation}
\begin{array}{ll}
\underset{x\in X}{\text{minimize}}\text{ } & \sum\limits_{i=1}^{m}[g_{i}%
\left( x\right) -b_{i}]_{+}^{p},%
\end{array}
\label{eq_problem_sum_slacks_k}
\end{equation}%
which also adapts to the format of (\ref{eq_sum_fi}). Sharper results are
presented for linear systems when $p=2.$

At this point, we summarize the structure of the paper. Section 2 gathers
some background on convex sets, convex functions, and monotone operators,
which is appealed to in the remaining sections. Section 3 explores some new
properties of paramonotone operators and, in particular, analyzes the
simultaneous fulfilment of paramonotonicity and bimonotonicity. The problem
of simultaneous projections -see (\ref{eq_min_dist_k_weighted}) and (\ref%
{eq_problem_u_in_Xm})- is tackled in Section 4, while the distance to
feasibility of convex systems under RHS perturbations is dealt with in
Section 5.

\section{Preliminaries}

Let $X$ be a real Banach space and $f:X\rightarrow \left] -\infty ,+\infty %
\right] $ be a proper lsc convex function. We denote by $\mathrm{dom}%
f:=\{x\in X\mid f\left( x\right) <+\infty \}$ the domain of function $f.$
Recall that the subdifferential operator of $f,$ $\partial
f:X\rightrightarrows X^{\ast },$ assigns to each $x\in \mathrm{dom}f$ the
(possibly empty) set $\partial f\left( x\right) $ formed by all $x^{\ast
}\in X^{\ast }$ (called subgradients) such that%
\begin{equation*}
f\left( y\right) -f\left( x\right) \geq \left\langle y-x,x^{\ast
}\right\rangle ,\text{ for all }y\in X.
\end{equation*}%
When $x\notin \mathrm{dom}f$ we define $\partial f\left( x\right)
:=\emptyset ;$ in this way the domain of the set-valued mapping, $\mathrm{dom%
}\partial f,$ is always contained in $\mathrm{dom}f.$ Associated with $f,$
its Fenchel conjugative function $f^{\ast }:X^{\ast }\rightarrow \left]
-\infty ,+\infty \right] $ is given by 
\begin{equation*}
f^{\ast }\left( x^{\ast }\right) =\sup \left\{ \left\langle x,x^{\ast
}\right\rangle -f\left( x\right) \mid x\in X\right\} .
\end{equation*}%
Recall that the Young-Fenchel inequality writes as $f^{\ast }\left( x^{\ast
}\right) +f\left( x\right) \geq \left\langle x,x^{\ast }\right\rangle $ for
all $x\in X.$

For completeness, we gather in the following theorem some well-known results
about $\partial f$ and $f^{\ast }$ in Banach spaces used in the paper. They
can be traced out from different references dealing with convex analysis in
infinite dimensional spaces. Here, we mainly cite the books \cite%
{BorweinVanderwerff10, Lucchetti06, MN22, Zalinescu02}. From now on, $%
\mathrm{int}A$ denotes the interior of $A\subset X$ (where, as usual, $%
\subset $ is understood as $\subseteq )$ and the zero vector of $X^{\ast }$
is denoted by just $0.$

\begin{theorem}
\label{Theorem_properties_convexity}Let $f:X\rightarrow \left] -\infty
,+\infty \right] $ be a proper lsc convex function. Then we have:

$\left( i\right) $ \emph{\cite[Proposition 4.1.5]{BorweinVanderwerff10}} $f$
is continuous at $x$ if and only if $x\in \mathrm{int~dom}f\mathrm{;}$

$\left( ii\right) $ \emph{\cite[Theorem 3.2.15]{Lucchetti06}} $\mathrm{%
int~dom}f\subset \mathrm{dom}\partial f;$

$\left( iii\right) $ \emph{\cite[Proposition 3.2.17]{Lucchetti06}} $x\in
\arg \min f$ if and only if $0\in \partial f\left( x\right) ;$

$\left( iv\right) $ \emph{\cite[Exercise 4.2.15]{Lucchetti06}} $\partial f$
is maximally monotone;

$\left( v\right) $ \emph{\cite[Fact 3.1]{BausWangYao10} (see \cite[%
Proposition 2.2]{Iusem98} for finite dimensions)} $\partial f$ is
paramonotone;

$\left( vi\right) $ \emph{\cite[Proposition 5.31]{Lucchetti06}} For any $%
x\in X,$ we have the equivalence 
\begin{equation*}
x^{\ast }\in \partial f\left( x\right) \Leftrightarrow f^{\ast }\left(
x^{\ast }\right) +f\left( x\right) =\left\langle x,x^{\ast }\right\rangle ;
\end{equation*}%
(indeed, this statement does not require convexity);

$\left( vii\right) $ \emph{\cite[Theorem 3.4.2]{Lucchetti06} (see also \cite[%
Theorem 3]{Rock66}) }Let $g:X\rightarrow \left] -\infty ,+\infty \right] $
be any proper convex function. If $\left( \mathrm{int~dom}f\right) \cap 
\mathrm{dom}g\neq \emptyset ,$ then we have the subdifferential sum rule%
\begin{equation*}
\partial f\left( x\right) +\partial g\left( x\right) =\partial \left(
f+g\right) \left( x\right) ,\text{ whenever }x\in \mathrm{dom}\partial f\cap 
\mathrm{dom}\partial g,
\end{equation*}%
(indeed, `$\supset $' is the nontrivial inclusion, as $\partial f\left(
x\right) +\partial g\left( x\right) \subset \partial \left( f+g\right)
\left( x\right) $ comes directly from the definition of subdifferential;
moreover, the lower semicontinuity of $f$ is not needed).
\end{theorem}

Recall that, for an arbitrary monotone operator $T:X\rightrightarrows
X^{\ast },$ an lsc convex function $h:X\times X^{\ast }\rightarrow \left]
-\infty ,+\infty \right] $ is said to be a \emph{representative} \emph{%
function }of $T$ if 
\begin{equation*}
h(x,x^{\ast })\left\{ 
\begin{array}{cl}
=\left\langle x,x^{\ast }\right\rangle , & \text{if }(x,x^{\ast })\in \gph T,
\\ 
>\left\langle x,x^{\ast }\right\rangle , & \text{elsewhere.}%
\end{array}%
\right. .
\end{equation*}%
Operators for which a representative function exists are called
representable monotone. For a detailed study of representable monotone
operators we refer to \cite{MS05}, where this notion was introduced.

\begin{remark}
\emph{From Theorem \ref{Theorem_properties_convexity}}$\left( vi\right) ,$%
\emph{\ observe that, if }$f:X\rightarrow \left] -\infty ,+\infty \right] $ 
\emph{is a proper lsc convex function, the function} $h_{f}:X\times X^{\ast
}\rightarrow \left] -\infty ,+\infty \right] $ \emph{defined by} 
\begin{equation}
h_{f}(x,x^{\ast })=f\left( x\right) +f^{\ast }\left( x^{\ast }\right) ,\text{
\emph{for }}\left( x,x^{\ast }\right) \in X\times X^{\ast }  \label{eq_hf}
\end{equation}%
\emph{is a representative function of }$\partial f.$ \emph{More generally,
every maximally monotone operator is representable, as far as its well-known
Fitzpatrick function is a representative function (see, e.g. \cite[Section
9.1.2]{BorweinVanderwerff10} for details). An easy consequence of this fact
is that intersections of arbitrary collections of maximally monotone
operators are representable, too. According to \cite[Corollary 32]{MS05}, in
finite-dimensional spaces only such intersections are representable. This is
no longer true in infinite dimensional spaces, as proved in \cite[Theorem
11.2]{S07}.}
\end{remark}

The rest of this section is devoted to recall some results about metric
projections and, in order to ensure existence and uniqueness of the best
approximation to closed convex sets, we assume that $X$ is a Hilbert space.
Here, $\left\Vert \cdot \right\Vert $ denotes the norm associated with the
corresponding inner product $\left\langle \cdot ,\cdot \right\rangle $.
Given any nonempty closed convex set $S\subset X,$ we denote by $%
P_{S}:X\rightarrow X$ the \emph{metric projection on }$S$, which assigns to
each $x\in X$ its (unique) best approximation in $S,$ denoted by $%
P_{S}\left( x\right) $, i.e., $P_{S}\left( x\right) $ is the unique point of 
$S$ such that 
\begin{equation*}
\left\Vert x-P_{S}\left( x\right) \right\Vert =d\left( x,S\right) =\min
\left\{ \left\Vert x-s\right\Vert :s\in S\right\} .
\end{equation*}%
(Observe that we write $P_{S}:X\rightarrow X$ instead $P_{S}:X%
\rightrightarrows X$ due to its single-valuedness.) It is well-known that
function $x\mapsto d\left( x,S\right) ,$ denoted for convenience by $%
d_{S}:X\rightarrow \left[ 0,+\infty \right[ ,$ is a continuous convex
function. Recall that, for a continuous convex function, $f:X\rightarrow 
\mathbb{R},$ applying \cite[Corollary 4.2.5 ]{BorweinVanderwerff10}, we
deduce that $f$ is G\^{a}teaux differentiable at a point $x$ if and only if $%
\partial f\left( x\right) $ reduces to a singleton, i.e. $\partial f\left(
x\right) =\left\{ \nabla f\left( x\right) \right\} $; see \cite[Section 2 ]%
{BorweinVanderwerff10} for details. In our applications, the facts that the
subdifferentials $\partial d_{S}\left( x\right) $ or $\partial
d_{S}^{2}\left( x\right) $ reduce to a singleton are crucial. Accordingly,
condition $\left( i\right) $ in the following proposition is stated directly
in these terms (instead of G\^{a}teaux differentiability). From now on, $%
N_{S}\left( x\right) $ denotes the normal cone to $S$ at $x$ which is given
by 
\begin{equation}
N_{S}\left( x\right) :=\left\{ x^{\ast }\in X^{\ast }\mid \left\langle
s-x,x^{\ast }\right\rangle \leq 0,\text{ }s\in S\right\} ,
\label{eq_normal_cone}
\end{equation}%
$B^{\ast }$ denotes the closed unit ball in $X^{\ast }$ and $\mathrm{bd}S$
the boundary of $S$.

\begin{proposition}
\label{Prop_dist_diff}\label{Prop_distance2}Let $X$ be a Hilbert space and $%
\emptyset \neq S\subset X$ a closed set. Then, we have

$\left( i\right) $ \emph{\cite[Corollary 4.2.5 and Theorem 4.5.7]%
{BorweinVanderwerff10} }$S$ is convex if and only if $\partial
d_{S}^{2}\left( x\right) $ is singleton for all $x\in X;$ in such a case%
\textbf{,} 
\begin{equation*}
\nabla d_{S}^{2}\left( x\right) =2\left( x-P_{S}\left( x\right) \right) .
\end{equation*}

$\left( ii\right) $ \emph{\cite[Proposition 4.1.5]{Lucchetti06}} \emph{(see
also \cite[Section 1]{holmes73})} If $S$ is convex, then 
\begin{equation*}
\partial d_{S}\left( x\right) =\left\{ 
\begin{array}{ll}
\left\{ 0\right\} , & \text{if }x\in \mathrm{int}S, \\ 
N_{S}\left( x\right) \cap B^{\ast },\text{ } & \text{if }x\in \mathrm{bd}S,
\\ 
\left\{ \left\Vert x-P_{S}\left( x\right) \right\Vert ^{-1}\left(
x-P_{S}\left( x\right) \right) \right\} , & \text{if }x\notin S.%
\end{array}%
\right.
\end{equation*}
\end{proposition}

\section{On paramonotone and bimonotone operators}

This section provides some results, appealed to in Sections 4 and 5, about
operators which are simultaneously paramonotone and bimonotone on a real
Banach space $X.$ To start with, we provide some basic results on these two
properties separately.

\begin{proposition}
\label{prop paramonotone repre}Let $T:X\rightrightarrows X^{\ast }$ be a
representable monotone operator. The following statements are equivalent:

$\left( i\right) $ $T$ is paramonotone;

$\left( ii\right) $ For any representative function $h$ of $T,$ the
following implication holds:%
\begin{equation}
\left. 
\begin{array}{c}
\left( x,x^{\ast }\right) ,\left( y,y^{\ast }\right) \in \mathrm{gph}T \\ 
\left\langle x-y,x^{\ast }-y^{\ast }\right\rangle =0%
\end{array}%
\right\} \Rightarrow h(x,y^{\ast })+h(y,x^{\ast })=h(x,x^{\ast
})+h(y,y^{\ast }).  \label{eq_implication_bis}
\end{equation}

$\left( iii\right) $There exists a representative function $h$ of $T$ such
that \emph{(\ref{eq_implication_bis})} holds.
\end{proposition}

\begin{proof}
$(i)\Rightarrow (ii)$. Consider any representative function of $T,$ $h,$ and
take $\left( x,x^{\ast }\right) ,\left( y,y^{\ast }\right) \in \mathrm{gph}%
T, $ with $\left\langle x-y,x^{\ast }-y^{\ast }\right\rangle =0$. Then, the
paramonotonicity entails $y^{\ast }\in T(x)$ and $x^{\ast }\in T(y),$
yielding 
\begin{equation*}
h(x,y^{\ast })+h(y,x^{\ast })=\left\langle x,y^{\ast }\right\rangle
+\left\langle y,x^{\ast }\right\rangle =\left\langle x,x^{\ast
}\right\rangle +\left\langle y,y^{\ast }\right\rangle =h(x,x^{\ast
})+h(y,y^{\ast });
\end{equation*}

$(ii)\Rightarrow (iii)$. Straightforward.

$(iii)\Rightarrow (i)$. Let $h$ be a representative function of $T$
satisfying (\ref{eq_implication_bis}). Let $x,y\in X$, $x^{\ast }\in T(x)$, $%
y^{\ast }\in T(y)$ and suppose $\left\langle x-y,x^{\ast }-y^{\ast
}\right\rangle =0.$ Hence, 
\begin{equation*}
h(x,y^{\ast })+h(y,x^{\ast })=h(x,x^{\ast })+h(y,y^{\ast })=\left\langle
x,x^{\ast }\right\rangle +\left\langle y,y^{\ast }\right\rangle
=\left\langle x,y^{\ast }\right\rangle +\left\langle y,x^{\ast
}\right\rangle .
\end{equation*}%
Since $h(x,y^{\ast })\geq \left\langle x,y^{\ast }\right\rangle $ and $%
h(y,x^{\ast })\geq \left\langle y,x^{\ast }\right\rangle ,$ these
inequalities actually hold as equalities, yielding $y^{\ast }\in T(x)$ and $%
x^{\ast }\in T(y)$.
\end{proof}

\begin{remark}
\emph{Observe that the paramonotonicity of the subdifferential operator }$%
\partial f$ \emph{of a proper lsc convex function }$f$\emph{\ can be
alternatively deduced from Proposition \ref{prop paramonotone repre}.} \emph{%
Just consider the representative function} $h_{f}$ \emph{introduced in (\ref%
{eq_hf}), which is separable and, hence, one always has}%
\begin{equation*}
h_{f}(x,y^{\ast })+h_{f}(y,x^{\ast })=h_{f}(x,x^{\ast })+h_{f}(y,y^{\ast }).
\end{equation*}
\end{remark}

Other examples of paramonotone operators are mappings of the form $I-A$
where $I$ is the identity mapping and $A$ is nonexpansive (see \cite[Theorem
6.1]{BauschXWangLYao14}); see also \cite[Section 3 ]{Iusem98} for the
analysis of paramonotonicity of affine functions in $\mathbb{R}^{n}.$

\begin{proposition}
\label{prop bimonotone char}For an operator $T:X\rightrightarrows X^{\ast },$
the following conditions are equivalent:

$\left( i\right) $ $T$ is bimonotone;

$\left( ii\right) $ $\left\langle x-y,x^{\ast }-y^{\ast }\right\rangle =0,$
whenever $\left( x,x^{\ast }\right) ,\left( y,y^{\ast }\right) \in \mathrm{%
gph}T;$

$\left( iii\right) $ There exist monotone operators $T_{1}$ and $T_{2}$ such
that $T=T_{1}\cap \left( -T_{2}\right) .$
\end{proposition}

\begin{proof}
$\left( i\right) \Leftrightarrow \left( ii\right) $ is trivial.

$\left( i\right) \Rightarrow \left( iii\right) $. Write $T=T\cap \left(
-\left( -T\right) \right) $.

$\left( iii\right) \Rightarrow \left( i\right) $. If $T_{1}$ is monotone, so
is $T,$ since $T\subset T_{1}$. Analogously, since%
\begin{equation*}
-T=T_{2}\cap \left( -T_{1}\right) ,
\end{equation*}%
$-T$ is monotone, too.
\end{proof}

\begin{proposition}
\label{prop bimonotone repre}For a representable monotone operator $%
T:X\rightrightarrows X^{\ast },$ the following conditions are equivalent:

$\left( i\right) $ $T$ is bimonotone;

$\left( ii\right) $ For any representative function $h$ of $T$, the
following implication holds%
\begin{equation}
h(x,x^{\ast })+h(y,y^{\ast })=\left\langle x,x^{\ast }\right\rangle
+\left\langle y,y^{\ast }\right\rangle \Rightarrow h(x,x^{\ast
})+h(y,y^{\ast })=\left\langle x,y^{\ast }\right\rangle +\left\langle
y,x^{\ast }\right\rangle .  \label{impl bim}
\end{equation}

$(iii)$ There exists a representative function $h$ of $T$ such that \emph{(%
\ref{impl bim})} holds.
\end{proposition}

\begin{proof}
$\left( i\right) \Rightarrow \left( ii\right) .$ Consider any representative
function $h$ of $T$, and assume that 
\begin{equation*}
h(x,x^{\ast })+h(y,y^{\ast })=\left\langle x,x^{\ast }\right\rangle
+\left\langle y,y^{\ast }\right\rangle .
\end{equation*}%
Hence, $h(x,x^{\ast })=\left\langle x,x^{\ast }\right\rangle $ and $%
h(y,y^{\ast })=\left\langle y,y^{\ast }\right\rangle ,$ that is, $\left(
x,x^{\ast }\right) ,\left( y,y^{\ast }\right) \in \mathrm{gph}T,$ yielding $%
\left\langle x-y,x^{\ast }-y^{\ast }\right\rangle =0.$ Consequently,%
\begin{equation*}
\left\langle x,y^{\ast }\right\rangle +\left\langle y,x^{\ast }\right\rangle
=\left\langle x,x^{\ast }\right\rangle +\left\langle y,y^{\ast
}\right\rangle =h(x,x^{\ast })+h(y,y^{\ast }).
\end{equation*}

$\left( ii\right) \Rightarrow \left( iii\right) .$ Straightforward.

$\left( iii\right) \Rightarrow \left( i\right) .$ Let $x,y\in X,$ $\left(
x,x^{\ast }\right) ,\left( y,y^{\ast }\right) \in \mathrm{gph}T.$ We then
have%
\begin{equation*}
h(x,x^{\ast })+h(y,y^{\ast })=\left\langle x,x^{\ast }\right\rangle
+\left\langle y,y^{\ast }\right\rangle ,
\end{equation*}%
and hence $\left\langle x,y^{\ast }\right\rangle +\left\langle y,x^{\ast
}\right\rangle =h(x,x^{\ast })+h(y,y^{\ast })=\left\langle x,x^{\ast
}\right\rangle +\left\langle y,y^{\ast }\right\rangle ,$ from which the
equality $\left\langle x-y,x^{\ast }-y^{\ast }\right\rangle =0$ immediately
follows.
\end{proof}

\bigskip

From now on, symbol `$\perp $' represents orthogonality; specifically, given
any subsets $A\subset X$ and $B\subset X^{\ast },$ $A\perp B$ means that $%
\left\langle x,x^{\ast }\right\rangle =0$ for any $\left( x,x^{\ast }\right)
\in A\times B,$ whereas $A^{\bot }:=\left\{ x^{\ast }\in X^{\ast }\mid
\left\langle x,x^{\ast }\right\rangle =0,\text{ for all }x\in A\right\} $
and $B^{\bot }:=\{x\in X\mid \allowbreak \left\langle x,x^{\ast
}\right\rangle =0,$ for all $x^{\ast }\in B\}.$

\begin{corollary}
\label{cor bimon param implies constant}For $T:X\rightrightarrows X^{\ast },$
the following conditions are equivalent:

$\left( i\right) $ $T$ is paramonotone and bimonotone$;$

$\left( ii\right) $ $T$ is monotone and constant on its domain$;$

$\left( iii\right) $ $\left( \dom T-\dom T\right) \bot \left( \range T-%
\range
T\right) $ and $\mathrm{gph}~T=\dom T\times \range T.$
\end{corollary}

\begin{proof}
$\left( i\right) \Rightarrow \left( ii\right) $. Let $x,y\in \dom T,$ and
take $x^{\ast }\in T\left( x\right) ,\,\ y^{\ast }\in T\left( y\right) .$ By
bimonotonicity, we have $\left\langle x-y,x^{\ast }-y^{\ast }\right\rangle
=0.$ Hence, by paramonotonicity, $y^{\ast }\in T\left( x\right) $\thinspace
and$\ x^{\ast }\in T\left( y\right) .$ This proves that $T\left( x\right)
=T\left( y\right) .$

$\left( ii\right) \Rightarrow \left( i\right) $. The paramonotonicity of $T$
is an obvious consequence of its being constant on its domain. To prove
bimonotonicity, let $\left( x,x^{\ast }\right) ,\left( y,y^{\ast }\right)
\in \mathrm{gph}T$. Monotonicity implies $\left\langle x-y,x^{\ast }-y^{\ast
}\right\rangle \geq 0$, and we can interchange $x^{\ast }$ and $y^{\ast },$
since $T\left( x\right) =T\left( y\right) ;$ therefore $\left\langle
x-y,x^{\ast }-y^{\ast }\right\rangle =0.$

$\left( ii\right) \Leftrightarrow \left( iii\right) $ Comes from the fact
that $T$ is constant on $\mathrm{dom}T$ if and only if $\mathrm{gph}~T=\dom %
T\times \range T$.
\end{proof}

\begin{corollary}
Let $T:X\rightrightarrows X^{\ast }$ be paramonotone and bimonotone. Then,
we have

$\left( i\right) $ If $\dom T$ is dense in $X$, then $T$ is single valued$;$

$\left( ii\right) $ If $\range T$ is dense in $X^{\ast }$, then $\dom T$ is
a singleton$;$

$\left( iii\right) $ $T$ is maximally monotone if and only if $\dom T$ and $%
\range T$ are closed affine varieties and 
\begin{equation}
\dom T-\dom T=\left( \range T-\range T\right) ^{\perp }.  \label{orth}
\end{equation}
\end{corollary}

\begin{proof}
$\left( i\right) $ Assume, reasoning by contradiction, that there exist $%
\left( x,x^{\ast }\right) $ and\textbf{\ }$\left( x,\widetilde{x}^{\ast
}\right) $ in $\mathrm{gph}T$ with $x^{\ast }\neq \widetilde{x}^{\ast }$ and
take $u\in X$ with $\left\langle u,x^{\ast }-\widetilde{x}^{\ast
}\right\rangle \neq 0.$ Under the current assumption, we can take \text{a
sequence }$\{x^{r}\}_{r\in \mathbb{N}}\subset \dom T$ converging to $x+u.$
For $r$ large enough we have $\left\langle x^{r}-x,x^{\ast }-\widetilde{x}%
^{\ast }\right\rangle \neq 0$ and $x^{\ast }\in T\left( x\right) =T\left(
x^{r}\right) $ because of Corollary \ref{cor bimon param implies constant}.
This contradicts $\left( i\right) \Rightarrow \left( ii\right) $ in
Proposition \ref{prop bimonotone char}.

$\left( ii\right) $ follows analogously to $\left( i\right) $ by considering 
$\left( x,x^{\ast }\right) $ and $\left( \widetilde{x},x^{\ast }\right) $ in 
$\mathrm{gph}T$ with $x\neq \widetilde{x},$ taking again Corollary\textbf{\ }%
\ref{cor bimon param implies constant} into account.

$\left( iii\right) $ Assume that $T$ is maximally monotone. Take $x_{0}\in %
\dom T,$ $x_{0}^{\ast }\in \range T,$ and let $S$ and $S_{\ast }$ be the
linear subspaces generated by $\dom T-\dom T$ and $\range T-\range T,$
respectively. Define $\widehat{T}:X\rightrightarrows X^{\ast }$ by%
\begin{equation*}
\widehat{T}\left( x\right) :=\left\{ 
\begin{array}{c}
x_{0}^{\ast }+S^{\bot }\text{ if }x\in x_{0}+\cl S, \\ 
\emptyset \text{ \qquad otherwise. \qquad }%
\end{array}%
\right.
\end{equation*}%
We have%
\begin{eqnarray*}
\dom\widehat{T}-\dom\widehat{T} &=&\left( x_{0}+\cl S\right) -\left( x_{0}+%
\cl S\right) =\cl S-\cl S=\cl S=\left( S^{\bot }\right) ^{\bot } \\
&=&\left( S^{\bot }-S^{\bot }\right) ^{\bot }=\left( \left( x_{0}^{\ast
}+S^{\bot }\right) -\left( x_{0}^{\ast }+S^{\bot }\right) \right) ^{\bot } \\
&=&\left( \range\widehat{T}-\range\widehat{T}\right) ^{\perp },
\end{eqnarray*}%
which proves (\ref{orth}) for operator $\widehat{T}$. Moreover,%
\begin{equation*}
\mathrm{gph}~\widehat{T}=\left( x_{0}+\cl S\right) \times \left( x_{0}^{\ast
}+S^{\bot }\right) =\dom\widehat{T}\times \range\widehat{T}.
\end{equation*}%
Therefore, by equivalence $\left( i\right) \Leftrightarrow \left( iii\right) 
$ in Corollary \ref{cor bimon param implies constant}, the operator $%
\widehat{T}$ is paramonotone and bimonotone; in particular, $\widehat{T}$ is
monotone. On the other hand, by the same equivalence, we have%
\begin{equation*}
\range T-\range T\subset \left( \dom T-\dom T\right) ^{\bot }=S^{\bot };
\end{equation*}%
hence $S_{\ast }\subset S^{\bot }$ and%
\begin{eqnarray*}
\mathrm{gph}~T &=&\dom T\times \range T\subset \left( x_{0}+\cl S\right)
\times \left( x_{0}^{\ast }+\cl S_{\ast }\right) \subset \dom\widehat{T}%
\times \left( x_{0}^{\ast }+S^{\bot }\right) \\
&=&\dom\widehat{T}\times \range\widehat{T}=\mathrm{gph}~\widehat{T}.
\end{eqnarray*}%
Thus, by the maximal monotonicity of $T,$ we have $T=\widehat{T},$ from
which we deduce that $\dom T=\dom\widehat{T}=x_{0}+\cl S$ and $\range T=%
\range\widehat{T}=x_{0}^{\ast }+S^{\bot },$ thus proving that $\dom T$ and $%
\range T$ are closed affine varieties.

Let us see the converse implication. Let $(x,x^{\ast })\in X\times X^{\ast }$
be such that%
\begin{equation}
\left\langle x-y,x^{\ast }-y^{\ast }\right\rangle \geq 0\text{ for all }%
\left( y,y^{\ast }\right) \in \mathrm{gph}~T=\dom T\times \range T
\label{eq_mm}
\end{equation}%
(the latter equality following again from Corollary \ref{cor bimon param
implies constant}). Since $\range T$ is an affine variety, we can easily
prove that $x-y\in \left( \range T-\range T\right) ^{\bot }=\dom T-\dom T.$
More in detail, replace $y^{\ast }$\textbf{\ }in (\ref{eq_mm}) with\textbf{\ 
}$x_{0}^{\ast }\pm \lambda v$\textbf{\ }for any given $x_{0}^{\ast }\in %
\range T$ and any\textbf{\ }$v\in \range T-\range T,$ with $\lambda >0,$%
\textbf{\ }then divide both sides of the resulting specification of \textbf{(%
\ref{eq_mm}) }by $\lambda $\ and let $\lambda \rightarrow +\infty $ to obtain%
\textbf{\ }$\left\langle x-y,\pm v\right\rangle \geq 0.$\textbf{\ }%
Therefore, given that $\dom T$ is an affine variety, we deduce that $x\in %
\dom T.$ Similarly, using that $\range T-\range T=\left( \dom T-\dom %
T\right) ^{\bot },$ we obtain that $x^{\ast }\in \range T.$ Thus, $%
(x,x^{\ast })\in \dom T\times \range T=\mathrm{gph}~T,$ which proves that $T$
is maximally monotone.
\end{proof}

\bigskip

The following propositions involve a finite number of paramonotone operators
and are intended to provide a unified framework to deal with the
applications of Sections 4 and 5. First, we introduce the following lemma,
which has an easy proof.

\begin{lemma}
\label{sum}If $T_{1}:X\rightrightarrows X^{\ast }$ and $T_{2}:X%
\rightrightarrows X^{\ast }$ are paramonotone, then so are $T_{1}+T_{2}$ and 
$T_{1}\cap \left( -T_{2}\right) .$
\end{lemma}

\begin{proposition}
\label{prop 2 param intersec}\label{Th_intersection}Let $T_{i}:X%
\rightrightarrows X^{\ast },$ $i=1,\ldots ,m,$ be paramonotone operators.
Then the intersection mappings%
\begin{equation}
\widetilde{T}_{i}:=T_{i}\cap \left( -\medop\sum_{j\neq i}T_{j}\right) ,\text{
}i=1,\ldots ,m.  \label{eq_1}
\end{equation}%
are monotone and constant in their common domain%
\begin{equation*}
\mathcal{A}:\mathcal{=}\left\{ x\in X\mid 0\in \medop\sum_{j=1}^{m}T_{j}%
\left( x\right) \right\} .
\end{equation*}
\end{proposition}

\begin{proof}
Fix $i\in \{1,\ldots ,m\}.\,\ $From Lemma \ref{sum}, $\sum_{j\neq i}T_{j}$
is paramonotone and, hence, the same lemma establishes that $\widetilde{T}%
_{i}$ is paramonotone. Then, from equivalence $\left( i\right)
\Leftrightarrow \left( iii\right) $ in Proposition \ref{prop bimonotone char}%
, $\widetilde{T}_{i}$ is bimonotone. Hence \ref{cor bimon param implies
constant} $\left( ii\right) $ yields that $\widetilde{T}_{i}$ is constant in 
$\mathrm{dom}\widetilde{T}_{i}.$ Finally, one easily sees that $\mathrm{dom}%
\widetilde{T}_{i}$ coincides with $\mathcal{A}$.
\end{proof}

\bigskip

Now, we particularize Proposition \ref{prop 2 param intersec} by considering
finitely many proper lsc convex functions, $f_{i}:X\rightarrow \left]
-\infty ,\mathbb{+\infty }\right] ,$ $i=1,\ldots ,m,$ and the corresponding
subdifferential operators $T_{i}:=\partial f_{i},$ $i=1,\ldots ,m.$ We
assume the following regularity condition in order to apply the
subdifferential sum rule (see Theorem \ref{Theorem_properties_convexity} $%
\left( vii\right) )$: there exists some index $i_{0}\in \{1,\ldots ,m\},$
such that%
\begin{equation}
\mathrm{dom}f_{i_{0}}\cap \left( \tbigcap_{i\neq i_{0}}\mathrm{intdom}%
f_{i}\right) \neq \emptyset ,  \label{eq_regul_cond}
\end{equation}%
which is equivalent to the existence of some $\overline{x}\in \cap
_{i=1,\ldots ,m}\mathrm{dom}f_{i}$ such that the $m-1$ of the functions $%
f_{i},$ $i\in \{1,\ldots ,m\}\setminus \{i_{0}\}$ are continuous at $%
\overline{x}$ (see Theorem \ref{Theorem_properties_convexity} $\left(
i\right) ).$

In this particular case, we are considering the operators 
\begin{equation}
\widetilde{T}_{i}:=\partial f_{i}\cap \left( -\sum_{j\neq i}\partial
f_{j}\right) ,\text{ }i=1,\ldots ,m,  \label{eq_2}
\end{equation}%
whose the common domain, appealing to statements $\left( iii\right) $ and $%
\left( vii\right) $ in Theorem \ref{Theorem_properties_convexity}, writes as%
\begin{equation}
\mathcal{A}\mathcal{=}\left\{ x\in X\mid 0\in \sum_{i=1}^{m}\partial
f_{i}\left( x\right) \right\} =\arg \min \sum_{i=1}^{m}f_{i}.  \label{eq_3}
\end{equation}

We summarize the previous comments in the following proposition.

\begin{proposition}
\label{op const}Let $f_{i}:X\rightarrow \left] -\infty ,\mathbb{+\infty }%
\right] ,$ $i=1,\ldots ,m,$ be proper lsc convex functions and assume that
for some $i_{0}\in \{1,\ldots ,m\}$ condition (\ref{eq_regul_cond}) holds.
Then, operators $\widetilde{T}_{i},$ $\{1,\ldots ,m\},$ defined in (\ref%
{eq_2}) are constant on their common domain 
\begin{equation*}
\mathcal{A=}\arg \min \sum_{i=1}^{m}f_{i}.
\end{equation*}
\end{proposition}

\begin{remark}
\emph{Proposition \ref{op const} can be applied to specific operators in
order to derive some classical statements which can be found in the
literature, as the one of \cite[Lemma 2]{BurkeFerris91} involving }$\partial
f\cap \left( -N_{S}\right) $\emph{, and regarding the optimization problem}%
\begin{equation*}
\begin{array}{ll}
\text{minimize } & f\left( x\right) \\ 
\text{ subject to } & x\in C,%
\end{array}%
\end{equation*}%
\emph{in the case when }$f:\mathbb{R}^{n}\rightarrow \left] -\infty ,\mathbb{%
+\infty }\right] $ \emph{is a proper lsc convex function and }$C$ \emph{is a
closed convex subset of }$\mathbb{R}^{n}.$ \emph{Specifically, if }$S$ \emph{%
denotes the set of optimal solutions of such a problem, \cite[Lemma 2]%
{BurkeFerris91} states that }$\partial f\left( x\right) \cap \left(
-N_{C}\left( x\right) \right) $\emph{\ is independent of }$x\in S.$ \emph{To
derive this statement from Proposition \ref{op const}, just observe that the
normal cone operator, }$N_{C}$ \emph{(recall (\ref{eq_normal_cone})), is
paramonotone as it coincides with the subdifferential of the indicator
function of }$C.$
\end{remark}

\begin{corollary}
\label{Corollary_diff}\label{grad const}Under the assumptions of Proposition %
\ref{op const}, one has:

$\left( i\right) $ If for some $j_{0}\in \{1,\ldots ,m\},$ the function $%
f_{j_{0}}$ is differentiable at $\overline{x}\in \mathcal{A},$ then $%
\widetilde{T}_{j_{0}}\left( x\right) =\left\{ \nabla f_{j_{0}}(\overline{x}%
)\right\} ,$ for all $x\in \mathcal{A}.$

$\left( ii\right) $ If for some $j_{0}\in \{1,\ldots ,m\},$ the function $%
f_{j_{0}}$ is differentiable on $\mathcal{A},$ then $\nabla f_{i_{0}}$ is
constant on $\mathcal{A}.$
\end{corollary}

\begin{proof}
$\left( i\right) $ follows straightforwardly from Proposition \ref{op const}%
, taking into account that if $f_{j_{0}}$ is differentiable at $\overline{x}%
, $ then $\widetilde{T}_{j_{0}}\left( \overline{x}\right) =\partial
f_{j_{0}}\left( \overline{x}\right) =\left\{ \nabla f_{j_{0}}(\overline{x}%
)\right\} $ (since $\emptyset \neq \widetilde{T}_{j_{0}}\left( \overline{x}%
\right) \subset \left\{ \nabla f_{j_{0}}(\overline{x})\right\} $)$,$ which
entails that $\widetilde{T}_{j_{0}}\left( x\right) =\widetilde{T}%
_{j_{0}}\left( \overline{x}\right) =\left\{ \nabla f_{j_{0}}(\overline{x}%
)\right\} $ whenever $x\in \mathcal{A}.$

$\left( ii\right) $ comes from $\left( i\right) $ since for every $\overline{%
x},x\in \mathcal{A}$ we have 
\begin{equation*}
\left\{ \nabla f_{j_{0}}(\overline{x})\right\} =\widetilde{T}_{j_{0}}\left(
x\right) \subset \partial f_{j_{0}}\left( x\right) =\left\{ \nabla
f_{j_{0}}(x)\right\} ;
\end{equation*}%
hence $\nabla f_{j_{0}}(\overline{x})=\nabla f_{j_{0}}(x).$
\end{proof}

\section{Simultaneous projections and displacement mappings}

This section is mainly devoted to study the minimal weighted distance to two
disjoint non-empty closed and convex subsets $S_{1}$ and $S_{2}$ of a
Hilbert space $X$. We will denote by $d:X\times X\rightarrow \mathbb{R}$ the 
\emph{distance function} on $X,$ i.e., $d\left( x,y\right) :=\left\Vert
x-y\right\Vert ,$ and by $d_{S_{i}}:X\rightarrow \mathbb{R}$ the distance
function to $S_{i},$ $i=1,2.$ We set%
\begin{equation*}
d\left( S_{1},S_{2}\right) :=\inf_{s_{1}\in S_{1},\text{ }s_{2}\in
S_{2}}d\left( s_{1},s_{2}\right) .
\end{equation*}%
For arbitrary real numbers $\alpha _{1},\alpha _{2}>0$, with $\alpha
_{1}+\alpha _{2}=1,$ and $p\geq 1,$ we define%
\begin{eqnarray}
v(\alpha _{1},\alpha _{2},p) &:&=\inf_{x\in X}\alpha _{1}d\left(
x,S_{1}\right) ^{p}+\alpha _{2}d\left( x,S_{2}\right) ^{p},  \notag \\
\mathcal{A}(\alpha _{1},\alpha _{2},p) &:&=\arg \min \alpha
_{1}d_{S_{1}}^{p}+\alpha _{2}d_{S_{2}}^{p}.  \label{eq def A}
\end{eqnarray}

Observe that $v(\alpha _{1},\alpha _{2},p)$ and $\mathcal{A}(\alpha
_{1},\alpha _{2},p)$ are the optimal value and the set of optimal solutions
of problem (\ref{eq_min_dist_k_weighted}) for the case of two sets.\textbf{\ 
}Notice that $\mathcal{A}(\alpha _{1},\alpha _{2},p)$ may be empty;
consider, e.g., the case when $X:=\mathbb{R}^{2},$ $S_{1}$ is the convex
hull of a branch of a hyperbola and $S_{2}$ is one of its asymptotes; in
this case\textbf{\ }$v(\alpha _{1},\alpha _{2},p)=0$ is not attained.

We denote by $P_{1}:=P_{S_{1}}$ and $P_{2}:=P_{S_{2}}$ the metric
projections over $S_{1}$ and $S_{2},$ respectively. We distinguish several
cases depending on the values of the power $p$ and parameters $\alpha _{1}$
and $\alpha _{2}$. At this moment we advance that in the case when $%
d_{S_{1}}^{p}$ and $d_{S_{2}}^{p}$ are differentiable we are able to apply
Corollary \ref{Corollary_diff} to derive information about $\mathcal{A}%
(\alpha _{1},\alpha _{2},p).$ Going further, Proposition \ref{Prop_dist_diff}%
$\left( i\right) $ establishes the differentiability of $d_{S_{i}}^{p}$ on
the whole space $X$ when\textbf{\ }$p\geq 2,$ which allows us to tackle the
case of a finite amount of sets$.$

\bigskip

\textit{Case 1.\qquad }$p:=1,$ $\alpha _{1}\neq \alpha _{2}.$

Without loss of generality, we assume that $\alpha _{1}>\alpha _{2}.$ The
following result has a clear geometrical meaning according to Proposition %
\ref{equiv}.

\begin{proposition}
\label{k=1 alphas dif}If $\alpha _{1}>\alpha _{2},$ then $\mathcal{A}(\alpha
_{1},\alpha _{2},1)=\arg \min_{S_{1}}d_{S_{2}}.$
\end{proposition}

\begin{proof}
We start by proving that every $x\in X$ satisfies a useful inequality:%
\begin{eqnarray*}
\alpha _{1}d\left( P_{1}\left( x\right) ,S_{1}\right) +\alpha _{2}d\left(
P_{1}\left( x\right) ,S_{2}\right) &=&\alpha _{2}d\left( P_{1}\left(
x\right) ,S_{2}\right) \\
&\leq &\alpha _{2}\left( \left\Vert P_{1}\left( x\right) -x\right\Vert
+d\left( x,S_{2}\right) \right) \\
&=&\alpha _{2}\left( d\left( x,S_{1}\right) +d\left( x,S_{2}\right) \right)
\\
&\leq &\alpha _{1}d\left( x,S_{1}\right) +\alpha _{2}d\left( x,S_{2}\right) .
\end{eqnarray*}%
Since the latter inequality is strict when $x\notin S_{1},$ it follows that $%
\mathcal{A}(\alpha _{1},\alpha _{2},1)\subset S_{1}.$ To prove the inclusion 
$\mathcal{A}(\alpha _{1},\alpha _{2},1)\subset \arg \min_{S_{1}}d_{S_{2}},$
let $\overline{x}\in \mathcal{A}(\alpha _{1},\alpha _{2},1)$ and $x\in
S_{1}. $ Since $\overline{x}\in S_{1},$ we have%
\begin{equation*}
\alpha _{2}d\left( \overline{x},S_{2}\right) =\alpha _{1}d\left( \overline{x}%
,S_{1}\right) +\alpha _{2}d\left( \overline{x},S_{2}\right) \leq \alpha
_{1}d\left( x,S_{1}\right) +\alpha _{2}d\left( x,S_{2}\right) =\alpha
_{2}d\left( x,S_{2}\right) ,
\end{equation*}%
which shows that $\overline{x}\in \arg \min_{S_{1}}d_{S_{2}},$ thus proving
the desired inclusion. For the opposite inclusion, let $\overline{x}\in \arg
\min_{S_{1}}d_{S_{2}}$ and $x\in X.$ Then%
\begin{eqnarray*}
\alpha _{1}d\left( \overline{x},S_{1}\right) +\alpha _{2}d\left( \overline{x}%
,S_{2}\right) &=&\alpha _{2}d\left( \overline{x},S_{2}\right) \leq \alpha
_{2}d\left( P_{1}\left( x\right) ,S_{2}\right) \\
&=&\alpha _{1}d\left( P_{1}\left( x\right) ,S_{1}\right) +\alpha _{2}d\left(
P_{1}\left( x\right) ,S_{2}\right) \\
&\leq &\alpha _{1}d\left( x,S_{1}\right) +\alpha _{2}d\left( x,S_{2}\right) ,
\end{eqnarray*}%
which implies that $\overline{x}\in \mathcal{A}(\alpha _{1},\alpha _{2},1).$
Therefore $\arg \min_{S_{1}}d_{S_{2}}\subset \mathcal{A}(\alpha _{1},\alpha
_{2},1),$ so the equality in the statement is proved.
\end{proof}

\bigskip

In the following corollary, $\Pi _{1}:S_{1}\times S_{2}\rightarrow S_{1}$
denotes de projection mapping, defined by $\Pi _{1}\left( s_{1},s_{2}\right)
=s_{1}.$

\begin{corollary}
If $\alpha _{1}>\alpha _{2},$ then $\mathcal{A}(\alpha _{1},\alpha
_{2},1)=\Pi _{1}\left( \arg \min_{S_{1}\times S_{2}}d\right) $
\end{corollary}

\begin{proof}
Taking into account Proposition \ref{k=1 alphas dif}, we will actually prove
the equivalent equality $\arg \min_{S_{1}}d_{S_{2}}=\Pi _{1}\left( \arg
\min_{S_{1}\times S_{2}}d\right) .$ To prove the inclusion $\subset ,$ let $%
\overline{s}_{1}\in \arg \min_{S_{1}}d_{S_{2}}$ and $s_{1}\in S_{1}.$ Then,
for every $s_{2}\in S_{2},$ we have%
\begin{equation*}
d\left( \overline{s}_{1},P_{2}\left( \overline{s}_{1}\right) \right)
=d\left( \overline{s}_{1},S_{2}\right) \leq d\left( s_{1},S_{2}\right) \leq
d\left( s_{1},s_{2}\right) ;
\end{equation*}%
hence $\left( \overline{s}_{1},P_{2}\left( \overline{s}_{1}\right) \right)
\in \arg \min_{S_{1}\times S_{2}}d,$ implying that $\overline{s}_{1}\in \Pi
_{1}\left( \arg \min_{S_{1}\times S_{2}}d\right) ,$ thus proving the desired
inclusion. We now proceed to prove the opposite inclusion. Let $\overline{s}%
_{1}\in \Pi _{1}\left( \arg \min_{S_{1}\times S_{2}}d\right) $ and $s_{1}\in
S_{1}.$ There exists $\overline{s}_{2}\in S_{2}$ such that $\left( \overline{%
s}_{1},\overline{s}_{2}\right) \in \arg \min_{S_{1}\times S_{2}}d,$ and for
every $s_{2}\in S_{2}$ we have%
\begin{equation*}
d\left( \overline{s}_{1},S_{2}\right) \leq d\left( \overline{s}_{1},%
\overline{s}_{2}\right) \leq d\left( s_{1},s_{2}\right) ;
\end{equation*}%
taking infimum over $s_{2}\in S_{2},$ this yields $d\left( \overline{s}%
_{1},S_{2}\right) \leq d\left( s_{1},S_{2}\right) ,$ which implies that $%
\overline{s}_{1}\in \arg \min_{S_{1}}d_{S_{2}}.$ Thus $\Pi _{1}\left( \arg
\min_{S_{1}\times S_{2}}d\right) \subset \arg \min_{S_{1}}d_{S_{2}},$ and
the proof is complete.
\end{proof}

\bigskip

\textit{Case 2.\qquad }$p=1,$ $\alpha _{1}=\alpha _{2}=\frac{1}{2}.$ From
now on $\left] P_{1}\left( x\right) ,P_{2}\left( x\right) \right[ $
represents the segment of points between $P_{1}\left( x\right) $ and $%
P_{2}\left( x\right) ,$ except these two ones; i.e.,\textbf{\ }$\left]
P_{1}\left( x\right) ,P_{2}\left( x\right) \right[ :=\left\{ \left(
1-\lambda \right) P_{1}\left( x\right) +\lambda P_{2}\left( x\right)
:0<\lambda <1\right\} .$

\begin{proposition}
One has:

$\left( i\right) $ $v(\frac{1}{2},\frac{1}{2},1)=\frac{1}{2}d\left(
S_{1},S_{2}\right) ,\medskip $

$\left( ii\right) $ $\mathcal{A}(\frac{1}{2},\frac{1}{2},1)=\left\{ x\in
X:x\in \left] P_{1}\left( x\right) ,P_{2}\left( x\right) \right[ \right\}
\cup \arg \min_{S_{1}}d_{S_{2}}\cup \arg \min_{S_{2}}d_{S_{1}}.$
\end{proposition}

\begin{proof}
$\left( i\right) $ For $x\in X,$ we have%
\begin{eqnarray*}
d\left( x,S_{1}\right) +d\left( x,S_{2}\right) &=&\left\Vert
x-P_{1}(x)\right\Vert +\left\Vert x-P_{2}(x)\right\Vert \geq \left\Vert
P_{1}(x)-P_{2}(x)\right\Vert \\
&\geq &d\left( S_{1},S_{2}\right) ,
\end{eqnarray*}%
which proves the inequality $\geq $. To prove the opposite inequality, it
suffices to observe that, for $s_{1}\in S_{1}$ and $s_{2}\in S_{2},$ we have%
\begin{eqnarray*}
\left\Vert s_{1}-s_{2}\right\Vert &=&d\left( s_{1},S_{1}\right) +\left\Vert
s_{1}-s_{2}\right\Vert \geq d\left( s_{1},S_{1}\right) +d\left(
s_{1},S_{2}\right) \\
&\geq &2v(\tfrac{1}{2},\tfrac{1}{2},1).
\end{eqnarray*}

$\left( ii\right) $ Let $x\in \mathcal{A}(\frac{1}{2},\frac{1}{2},1).$ If $%
x\notin S_{1}\cup S_{2},$ then $x\in \left] P_{1}\left( x\right)
,P_{2}\left( x\right) \right[ ,$ since otherwise we would have%
\begin{eqnarray*}
d\left( x,S_{1}\right) +d\left( x,S_{2}\right) &=&\left\Vert x-P_{1}\left(
x\right) \right\Vert +\left\Vert x-P_{2}\left( x\right) \right\Vert
>\left\Vert P_{1}\left( x\right) -P_{2}\left( x\right) \right\Vert \\
&\geq &d\left( S_{1},S_{2}\right) ,
\end{eqnarray*}%
a contradiction with $\left( i\right) $. If $x\in S_{1},$ then, for any%
\textbf{\ }$s_{1}\in S_{1}$\ we have%
\begin{equation*}
d\left( x,S_{2}\right) =d\left( x,S_{1}\right) +d\left( x,S_{2}\right) \leq
d\left( s_{1},S_{1}\right) +d\left( s_{1},S_{2}\right) =d\left(
s_{1},S_{2}\right) ,
\end{equation*}%
which shows that $x\in \arg \min_{S_{1}}d_{S_{2}}.$ In the same way, if $%
x\in S_{2},$ then $x\in \arg \min_{S_{2}}d_{S_{1}}.$ We have thus proved the
inclusion $\subset $. To prove the opposite inclusion, let $x\in X$ be such
that $x\in \left] P_{1}\left( x\right) ,P_{2}\left( x\right) \right[ $ and
take $\lambda \in \left] 0,1\right[ $ such that $x=(1-\lambda
)P_{1}(x)+\lambda P_{2}(x)$. Combining this equality with the inequalities $%
\left\langle s_{i}-P_{i}(x),x-P_{i}(x)\right\rangle \leq 0,$ which hold for
every $s_{i}\in S_{i},$ we obtain $\left\langle
s_{1}-P_{1}(x),P_{2}(x)-P_{1}(x)\right\rangle \leq 0$ and $\left\langle
s_{2}-P_{2}(x),P_{1}(x)-P_{2}(x)\right\rangle \leq 0.$ Adding the latter
inequalities, we get $\left\langle
s_{2}-s_{1}+P_{1}(x)-P_{2}(x),P_{1}(x)-P_{2}(x)\right\rangle \leq 0$;\ hence%
\begin{equation*}
\left\Vert P_{1}(x)-P_{2}(x)\right\Vert ^{2}\leq \left\langle
s_{1}-s_{2},P_{1}(x)-P_{2}(x)\right\rangle \leq \left\Vert
s_{1}-s_{2}\right\Vert \left\Vert P_{1}(x)-P_{2}(x)\right\Vert .
\end{equation*}%
Therefore, $\left\Vert s_{1}-s_{2}\right\Vert \geq \left\Vert
P_{1}(x)-P_{2}(x)\right\Vert $, and we deduce that%
\begin{equation*}
d\left( x,S_{1}\right) +d\left( x,S_{2}\right) =\left\Vert x-P_{1}\left(
x\right) \right\Vert +\left\Vert x-P_{2}\left( x\right) \right\Vert
=\left\Vert P_{1}(x)-P_{2}(x)\right\Vert \leq \left\Vert
s_{1}-s_{2}\right\Vert .
\end{equation*}%
Since $s_{i}\in S_{i},$ $i=1,2,$ are arbitrarily chosen, we conclude\textbf{%
\ }%
\begin{equation*}
d\left( x,S_{1}\right) +d\left( x,S_{2}\right) \leq d\left(
S_{1},S_{2}\right) ,
\end{equation*}%
which, by $\left( i\right) $, says that $x\in \mathcal{A}(\frac{1}{2},\frac{1%
}{2},1).$ It remains to prove that%
\begin{equation*}
\arg \min\nolimits_{S_{1}}d_{S_{2}}\cup \arg
\min\nolimits_{S_{2}}d_{S_{1}}\subset \mathcal{A}(\tfrac{1}{2},\tfrac{1}{2}%
,1).
\end{equation*}%
For symmetry reasons, it suffices to prove that $\arg
\min_{S_{1}}d_{S_{2}}\subset \mathcal{A}(\frac{1}{2},\frac{1}{2},1),$ but
this inclusion follows from the fact that, for $x\in \arg
\min_{S_{1}}d_{S_{2}},$ we have%
\begin{eqnarray*}
d\left( x,S_{1}\right) +d\left( x,S_{2}\right) &=&d\left( x,S_{2}\right)
=\min_{s_{1}\in S_{1}}d\left( s_{1},S_{2}\right) =\min_{s_{1}\in
S_{1}}\min_{s_{2}\in S_{2}}d\left( s_{1},s_{2}\right) \\
&=&\min_{s_{1}\in S_{1},\text{ }s_{2}\in S_{2}}d\left( s_{1},s_{2}\right)
=d\left( S_{1},S_{2}\right) .
\end{eqnarray*}
\end{proof}

\bigskip

\textit{Case 3.\qquad }$p>1.$

In our current setting, it is known that function $d_{S_{i}}:X\rightarrow 
\mathbb{R}$ is convex and differentiable outside $S_{i}$ and for every $x\in
X\setminus S_{i}$ one has (recall Proposition \ref{Prop_dist_diff}$\left(
ii\right) $) 
\begin{equation}
\nabla d_{S_{i}}\left( x\right) =\left( x-P_{i}(x)\right) /\left\Vert
x-P_{i}(x)\right\Vert .  \label{grad}
\end{equation}

\begin{theorem}
\label{A = fixed points}If $p>1,$ then:

$\left( i\right) $ $\mathcal{A}(\alpha _{1},\alpha _{2},p)\cap \left(
S_{1}\cup S_{2}\right) =\emptyset .$

$\left( ii\right) $ For each\textbf{\ }$i=1,2,$ function $d_{S_{i}}^{p}$ is
differentiable in $\mathcal{A}(\alpha _{1},\alpha _{2},p)$ and 
\begin{equation*}
\nabla d_{S_{i}}^{p}\left( x\right) =p\left\Vert x-P_{i}(x)\right\Vert
^{p-2}\left( x-P_{i}(x)\right) ,x\in \mathcal{A}(\alpha _{1},\alpha _{2},p).
\end{equation*}

$\left( iii\right) $ $\mathcal{A}(\alpha _{1},\alpha _{2},p)$ coincides with
the set of fixed points of 
\begin{equation*}
\frac{\alpha _{1}^{\frac{1}{p-1}}}{\alpha _{1}^{\frac{1}{p-1}}+\alpha _{2}^{%
\frac{1}{p-1}}}P_{1}+\frac{\alpha _{2}^{\frac{1}{p-1}}}{\alpha _{1}^{\frac{1%
}{p-1}}+\alpha _{2}^{\frac{1}{p-1}}}P_{2}.
\end{equation*}
\end{theorem}

\begin{proof}
$\left( i\right) $ It will suffice to prove that $\mathcal{A}(\alpha
_{1},\alpha _{2},p)\cap S_{1}=\emptyset .$ Let $x\in S_{1}$ and pick $%
\lambda >0$ such that $\frac{\lambda ^{p}}{1-(1-\lambda )^{p}}<\frac{\alpha
_{2}}{\alpha _{1}}$ (this is possible, since $\lim_{\lambda \rightarrow
0^{+}}\frac{\lambda ^{p}}{1-(1-\lambda )^{p}}=0$). Since%
\begin{equation*}
d\left( (1-\lambda )x+\lambda P_{2}(x),S_{1}\right) \leq d\left( (1-\lambda
)x+\lambda P_{2}(x),x\right) =\lambda \left\Vert P_{2}(x)-x\right\Vert
\end{equation*}%
and%
\begin{equation*}
d\left( (1-\lambda )x+\lambda P_{2}(x),S_{2}\right) \leq d\left( (1-\lambda
)x+\lambda P_{2}(x),P_{2}(x)\right) =(1-\lambda )\left\Vert
x-P_{2}(x)\right\Vert ,
\end{equation*}%
we have%
\begin{eqnarray*}
&&%
\begin{array}{c}
\alpha _{1}d\left( (1-\lambda )x+\lambda P_{2}(x),S_{1}\right) ^{p}+\alpha
_{2}d\left( (1-\lambda )x+\lambda P_{2}(x),S_{2}\right) ^{p}%
\end{array}
\\
&&%
\begin{array}{c}
\qquad \leq \alpha _{1}\lambda ^{p}\left\Vert P_{2}(x)-x\right\Vert
^{p}+\alpha _{2}\left( 1-\lambda \right) ^{p}\left\Vert
x-P_{2}(x)\right\Vert ^{p}%
\end{array}
\\
&&%
\begin{array}{c}
\qquad =\left( \alpha _{1}\lambda ^{p}+\alpha _{2}\left( 1-\lambda \right)
^{p}\right) \left\Vert x-P_{2}(x)\right\Vert ^{p}<\alpha _{2}\left\Vert
x-P_{2}(x)\right\Vert ^{p}%
\end{array}
\\
&&%
\begin{array}{c}
\qquad =\alpha _{1}d\left( x,S_{1}\right) ^{p}+\alpha _{2}d\left(
(x,S_{2}\right) ^{p},%
\end{array}%
\end{eqnarray*}%
which shows that $x\notin \mathcal{A}(\alpha _{1},\alpha _{2},p),$ thus
proving that $\mathcal{A}(\alpha _{1},\alpha _{2},p)$ and $S_{1}$ are
disjoint.\bigskip

$\left( ii\right) $\textbf{\ }is a consequence of $\left( i\right) $ taking (%
\ref{grad}) into account.

$\left( iii\right) $ For simplicity of notation, for $x\in \mathcal{A}%
(\alpha _{1},\alpha _{2},p)$ and $i=1,2$ we will denote 
\begin{equation}
D_{i}^{p}(x):=\alpha _{i}\nabla d_{S_{i}}^{p}\left( x\right) .
\label{grad pow}
\end{equation}%
Let $x\in \mathcal{A}(\alpha _{1},\alpha _{2},p).$ The equality $%
D_{1}^{p}(x)+D_{2}^{p}(x)=0$ yields%
\begin{equation}
\alpha _{1}\left\Vert x-P_{1}(x)\right\Vert ^{p-2}\left( x-P_{1}(x)\right)
+\alpha _{2}\left\Vert x-P_{2}(x)\right\Vert ^{p-2}\left( x-P_{2}(x)\right)
=0,  \label{fp0}
\end{equation}%
from which we deduce that%
\begin{eqnarray*}
x &=&\frac{\alpha _{1}\left\Vert x-P_{1}(x)\right\Vert ^{p-2}}{\alpha
_{1}\left\Vert x-P_{1}(x)\right\Vert ^{p-2}+\alpha _{2}\left\Vert
x-P_{2}(x)\right\Vert ^{p-2}}P_{1}(x) \\
&&+\frac{\alpha _{2}\left\Vert x-P_{2}(x)\right\Vert ^{p-2}}{\alpha
_{1}\left\Vert x-P_{1}(x)\right\Vert ^{p-2}+\alpha _{2}\left\Vert
x-P_{2}(x)\right\Vert ^{p-2}}P_{2}(x) \\
&=&\frac{\alpha _{1}}{\alpha _{1}+\alpha _{2}\left( \frac{\left\Vert
x-P_{2}(\left( x\right) x)\right\Vert }{\left\Vert x-P_{1}(x)\right\Vert }%
\right) ^{p-2}}P_{1}(x)+\frac{\alpha _{2}}{\alpha _{1}\left( \frac{%
\left\Vert x-P_{1}(x)\right\Vert }{\left\Vert x-P_{2}(x)\right\Vert }\right)
^{p-2}+\alpha _{2}}P_{2}(x)
\end{eqnarray*}%
Since condition $D_{1}^{p}(x)+D_{2}^{p}(x)=0$ implies that $\left\Vert
D_{1}^{p}(x)\right\Vert =\left\Vert D_{2}^{p}(x)\right\Vert ,$ that is, $%
\alpha _{1}p\left\Vert x-P_{1}(x)\right\Vert ^{p-1}=\alpha _{2}p\left\Vert
x-P_{2}(x)\right\Vert ^{p-1},$ which is equivalent to the equality%
\begin{equation}
\frac{\alpha _{1}}{\alpha _{2}}=\left( \frac{\left\Vert
x-P_{2}(x)\right\Vert }{\left\Vert x-P_{1}(x)\right\Vert }\right) ^{p-1},
\label{quot}
\end{equation}%
we obtain%
\begin{eqnarray*}
x &=&\frac{\alpha _{1}}{\alpha _{1}+\alpha _{2}\left( \frac{\alpha _{1}}{%
\alpha _{2}}\right) ^{\frac{p-2}{p-1}}}P_{1}(x)+\frac{\alpha _{2}}{\alpha
_{1}\left( \frac{\alpha _{2}}{\alpha _{1}}\right) ^{\frac{p-2}{p-1}}+\alpha
_{2}}P_{2}(x) \\
&=&\frac{\alpha _{1}^{\frac{1}{p-1}}}{\alpha _{1}^{\frac{1}{p-1}}+\alpha
_{2}^{\frac{1}{p-1}}}P_{1}(x)+\frac{\alpha _{2}^{\frac{1}{p-1}}}{\alpha
_{1}^{\frac{1}{p-1}}+\alpha _{2}^{\frac{1}{p-1}}}P_{2}(x).
\end{eqnarray*}%
This shows that $x$ is a fixed point of $\frac{\alpha _{1}^{\frac{1}{p-1}}}{%
\alpha _{1}^{\frac{1}{p-1}}+\alpha _{2}^{\frac{1}{p-1}}}P_{1}+\frac{\alpha
_{2}^{\frac{1}{p-1}}}{\alpha _{1}^{\frac{1}{p-1}}+\alpha _{2}^{\frac{1}{p-1}}%
}P_{2}.$

Conversely, if $x\in X$ is a fixed point of $\frac{\alpha _{1}^{\frac{1}{p-1}%
}}{\alpha _{1}^{\frac{1}{p-1}}+\alpha _{2}^{\frac{1}{p-1}}}P_{1}+\frac{%
\alpha _{2}^{\frac{1}{p-1}}}{\alpha _{1}^{\frac{1}{p-1}}+\alpha _{2}^{\frac{1%
}{p-1}}}P_{2},$ then $x\notin S_{1}\cup S_{2}.$ Indeed, otherwise, if, say, $%
x\in S_{1},$ then, from the equalities%
\begin{equation}
x=\frac{\alpha _{1}^{\frac{1}{p-1}}}{\alpha _{1}^{\frac{1}{p-1}}+\alpha
_{2}^{\frac{1}{p-1}}}P_{1}\left( x\right) +\frac{\alpha _{2}^{\frac{1}{p-1}}%
}{\alpha _{1}^{\frac{1}{p-1}}+\alpha _{2}^{\frac{1}{p-1}}}P_{2}\left(
x\right)  \label{fp}
\end{equation}%
and $P_{1}\left( x\right) =x$ we would obtain $x=P_{2}\left( x\right) \in
S_{2},$ thus contradicting the assumption that $S_{1}\cap S_{2}=\emptyset .$
Therefore, the functions $d_{S_{i}},$ $i=1,2,$ are differentiable at $x.$
From (\ref{fp}), it follows that 
\begin{equation}
\alpha _{1}^{\frac{1}{p-1}}\left( x-P_{1}\left( x\right) \right) +\alpha
_{2}^{\frac{1}{p-1}}\left( x-P_{2}\left( x\right) \right) =0,  \label{fp2}
\end{equation}%
from which we deduce (\ref{quot}). Now, using (\ref{quot}), we can rewrite (%
\ref{fp2}) as (\ref{fp0})\textbf{\ }to obtain the equality $%
D_{1}^{p}(x)+D_{2}^{p}(x)=0,$ which shows that $x\in \mathcal{A}(\alpha
_{1},\alpha _{2},p).$
\end{proof}

\bigskip

Notice that the set $\mathcal{A}(\frac{1}{2},\frac{1}{2},p)$ does not depend
on $p,$ since, by Theorem \ref{A = fixed points}$\left( iii\right) $, it
coincides with the set of fixed points of $\frac{1}{2}\left(
P_{1}+P_{2}\right) .$ Also notice that\textbf{\ }$\mathcal{A}(\alpha
_{1},\alpha _{2},2)$ coincides with the set of fixed points of $\alpha
_{1}P_{1}+\alpha _{2}P_{2}.$

The following lemma provides the counterpart of Theorem \ref{A = fixed
points}$\left( ii\right) $ for the case $p\geq 2.$

\begin{lemma}
\label{Lem_dist_K}Take $p\geq 2,$ and let $\emptyset \neq S\subset X$ be a
closed convex set$.$ The function $d_{S}^{p}$ is differentiable in $X$ and
we have%
\begin{equation*}
\nabla d_{S}^{p}\left( x\right) =pd_{S}^{p-2}\left( x\right) \left(
x-P_{S}\left( x\right) \right) ,\text{ for }x\in X.
\end{equation*}
\end{lemma}

\begin{proof}
Just write $d_{S}^{p}\left( x\right) $ as $\left( d_{S}^{2}\left( x\right)
\right) ^{p/2}$ and apply Proposition \ref{Prop_distance2} $\left( i\right) $%
.
\end{proof}

\bigskip

The fact that function $d_{S}^{p}$ is differentiable in the whole space $X$
enables us to tackle the case of a finite amount of subsets $%
S_{1},...,S_{m}, $ with $\cap _{i=1}^{m}S_{i}=\emptyset ,$ $m\in \mathbb{N}.$
For simplicity, we use the notation\textbf{\ } 
\begin{equation}
\mathcal{A}(\alpha ,p):=\arg \min \sum\limits_{i=1}^{m}\alpha
_{i}d_{S_{i}}^{p},  \label{eq_argmin_mSets}
\end{equation}%
where $\alpha :=\left( \alpha _{1},\alpha _{2},...,\alpha _{m}\right) ,$
with $\alpha _{i}>0,$ $i=1,...,m,$ and $\sum\limits_{i=1}^{m}\alpha _{i}=1$.
The following theorem gathers the announced application of Corollary \ref%
{grad const}$\left( ii\right) $.

\begin{theorem}
\label{Theorem_111}If $p>1$ and $m=2,$ or $p\geq 2,$ the displacement
mappings $I-P_{i},$ $i=1,...,m$ are constant on $\mathcal{A}(\alpha ,p).$
\end{theorem}

\begin{proof}
From Theorem\textbf{\ }\ref{A = fixed points}$\left( ii\right) $ and Lemma %
\ref{Lem_dist_K} if any of the current cases occurs we have that $%
d_{S_{i}}^{p}$ is differentiable on $\mathcal{A}(\alpha ,p),$ for each $%
i=1,...,m$. Hence, by Corollary \ref{grad const}$\left( ii\right) $, 
\begin{equation}
\nabla d_{S_{i}}^{p}\left( x\right) =p\left\Vert x-P_{i}(x)\right\Vert
^{p-2}\left( x-P_{i}\left( x\right) \right)  \label{eq_000}
\end{equation}%
is constant on $\mathcal{A}(\alpha ,p),$ $i=1,...,m$\textbf{\ }(again, $%
P_{i}:=P_{S_{i}},$ $i=1,2,...,m).$ So,%
\begin{equation*}
\left\Vert \nabla d_{S_{i}}^{p}(x)\right\Vert =p\left\Vert
x-P_{i}(x)\right\Vert ^{p-1}
\end{equation*}%
is constant on $\mathcal{A}(\alpha ,p),$ too,\textbf{\ }and hence so is $%
\left\Vert x-P_{i}(x)\right\Vert .$ Therefore, from (\ref{eq_000}), we
conclude that $I-P_{i}$ is constant on $\mathcal{A}(\alpha ,p),$ $i=1,...,m.$
\end{proof}

As a consequence of the previous theorem, taking Proposition \ref{equiv}
into account, we derive the following corollary. Roughly speaking, under the
current assumptions, the corollary says that the smallest translations of
the sets $S_{i}$ that achieve a nonempty intersection are unique.

\begin{corollary}
If $p>1$ and $m=2,$ or $p\geq 2,$ problem (\ref{eq_problem_u_in_Xm}) has a
unique optimal solution, provided that problem (\ref{eq_min_dist_k_weighted}%
) is solvable.\textbf{\ }
\end{corollary}

\section{Distance to feasibility}

This section is focused on the distance to feasibility for convex inequality
systems in $\mathbb{R}^{n}$ under RHS perturbations. In this framework,
lower and upper estimates for such a distance are provided in terms of some
elements whose existence is guaranteed from Corollary \ref{Corollary_diff}.
Both estimates coincide when confined to linear systems.

Let us consider the parameterized system$,$ 
\begin{equation}
\sigma \left( b\right) :=\left\{ g_{i}(x)\leq b_{i},\;i=1,\ldots ,m\right\} ,
\label{eq_convex:system}
\end{equation}%
where $x\in \mathbb{R}^{n},$ $(b_{i})_{i=1,\ldots ,m}\equiv b\in \mathbb{R}%
^{m},$ and $g_{i}:\mathbb{R}^{n}\rightarrow \mathbb{R}$ is a convex
function, $i=1,2,...,m$. To start with, the space of variables, $\mathbb{R}%
^{n},$ is endowed with an arbitrary norm, $\left\Vert \cdot \right\Vert ,$
with dual norm $\left\Vert \cdot \right\Vert _{\ast }$ and the associated
distances denoted by $d$ and $d_{\ast },$ respectively. From Corollary \ref%
{cor_Argmin} on we consider $\mathbb{R}^{n}$ equipped with the Euclidean
norm, $\left\Vert \cdot \right\Vert _{2}$. The space of parameters, $\mathbb{%
R}^{m}$, is endowed with any $p$-norm, $\left\Vert \cdot \right\Vert _{p},$
provided that $p\geq 2,$ and the associated distance is denoted by $d_{p}.$
We denote by $\Theta _{c}$ the set of consistent parameters; i.e., 
\begin{equation*}
\Theta _{c}:=\left\{ b\in \mathbb{R}^{m}\mid \sigma \left( b\right) \text{
is consistent}\right\} .
\end{equation*}%
Throughout this section we consider a fixed $\overline{b}\in \mathbb{R}%
^{m}\setminus \Theta _{c}$ and our aim is to estimate 
\begin{equation*}
d_{p}\left( \overline{b},\Theta _{c}\right) =\inf \left\{ \left\Vert 
\overline{b}-b\right\Vert _{p}:b\in \text{ is consistent}\right\} ,
\end{equation*}%
called the distance from $\overline{b}$ to feasibility.

\begin{proposition}
\label{Prop_distance_slack}Let $\overline{b}\in \mathbb{R}^{m}\setminus
\Theta _{c},$ then 
\begin{equation*}
d_{p}\left( \overline{b},\Theta _{c}\right) ^{p}=\inf_{x\in \mathbb{R}%
^{n}}\sum\limits_{i=1}^{m}[g_{i}\left( x\right) -\overline{b}_{i}]_{+}^{p}.
\end{equation*}
\end{proposition}

\begin{proof}
To establish the inequality `$\leq $',\ take any $x\in \mathbb{R}^{n}$ and
define%
\begin{equation*}
b_{i}:=\overline{b}_{i}+[g_{i}\left( x\right) -\overline{b}_{i}]_{+},\text{ }%
i=1,...,m.
\end{equation*}%
One can easily check that $b=\left( b_{i}\right) _{i=1,...,m}\in \Theta _{c}$
and, hence, 
\begin{equation*}
d_{p}\left( \overline{b},\Theta _{c}\right) ^{p}\leq d_{p}\left( \overline{b}%
,b\right) ^{p}=\sum\limits_{i=1}^{m}[g_{i}\left( x\right) -\overline{b}%
_{i}]_{+}^{p}.
\end{equation*}%
Since $x\in \mathbb{R}^{n}$ has been arbitrarily chosen, then 
\begin{equation*}
d_{p}\left( \overline{b},\Theta _{c}\right) ^{p}\leq \inf_{x\in \mathbb{R}%
^{n}}\sum\limits_{i=1}^{m}[g_{i}\left( x\right) -\overline{b}_{i}]_{+}^{p}.
\end{equation*}

Let us prove the converse inequality. Take any $b\in \Theta _{c}$, i.e.,
such that, for some $\overline{x}\in \mathbb{R}^{n},$ $g_{i}\left( \overline{%
x}\right) \leq b_{i},$ $i=1,...,m;$ then, $g_{i}\left( \overline{x}\right) -%
\overline{b}_{i}\leq b_{i}-\overline{b}_{i},$ $i=1,...,m,$ and so 
\begin{equation*}
\lbrack g_{i}\left( \overline{x}\right) -\overline{b}_{i}]_{+}\leq \lbrack
b_{i}-\overline{b}_{i}]_{+}\leq \left\vert b_{i}-\overline{b}_{i}\right\vert
,\text{ }i=1,...,m.
\end{equation*}%
Hence%
\begin{equation*}
\inf_{x\in \mathbb{R}^{n}}\sum\limits_{i=1}^{m}[g_{i}\left( x\right) -%
\overline{b}_{i}]_{+}^{p}\leq \sum\limits_{i=1}^{m}[g_{i}\left( \overline{x}%
\right) -\overline{b}_{i}]_{+}^{p}\leq \left\Vert \overline{b}-b\right\Vert
_{p}^{p}.
\end{equation*}%
Since the previous inequality is held for all $b\in \Theta _{c}$, then $%
\inf_{x\in \mathbb{R}^{n}}\sum\limits_{i=1}^{m}[g_{i}\left( x\right) -%
\overline{b}_{i}]_{+}^{p}\leq d_{p}\left( \overline{b},\Theta _{c}\right)
^{p}.$\bigskip
\end{proof}

The well-known Ascoli formula establishes that the distance from a point $%
x\in \mathbb{R}^{n}$ to a half-space $H:=\left\{ x\in \mathbb{R}^{n}\mid
\left\langle a,x\right\rangle \leq b\right\} ,$ with $0_{n}\neq a\in \mathbb{%
R}^{n}$ and $b\in \mathbb{R},$ is given by 
\begin{equation}
d_{H}\left( x\right) =\dfrac{[\left\langle a,x\right\rangle -b]_{+}}{%
\left\Vert a\right\Vert _{\ast }}.  \label{eq_ascoli}
\end{equation}%
The following result is focused on the extension of (\ref{eq_ascoli}) to the
convex case, where a convex inequality of the form `$g(x)\leq b$' is
considered. In this context, the distance from $x\in \mathbb{R}^{n}$ to the
nonempty closed convex set $S:=\left\{ x\in \mathbb{R}^{n}\mid g(x)\leq
b\right\} ,$ denoted by $d_{S}\left( x\right) ,$ is lower and upper bounded
by quotients involving the residual $[g\left( x\right) -b]_{+}$ and the
minimum norm of some subgradients of $g.$ Regarding these quotients, we use
the convention $\frac{0}{0}:=0.$

\begin{proposition}
\label{Prop generalized Ascoli}Let $g:\mathbb{R}^{n}\rightarrow \mathbb{R}$
be a convex function and $b\in \mathbb{R}$ such that the corresponding
sublevel set, $S$, is nonempty. Then we have:

$\left( i\right) $ For any $x\in \mathbb{R}^{n},$%
\begin{equation*}
d_{S}\left( x\right) \geq \dfrac{\lbrack g\left( x\right) -b]_{+}}{d_{\ast
}\left( 0_{n},\partial g\left( x\right) \right) };
\end{equation*}

$\left( ii\right) $ Assume that there exists $\widehat{x}\in \mathbb{R}^{n}$
(called a Slater point) such that $g(\widehat{x})<b.$ Then, for any $x\in 
\mathbb{R}^{n},$ 
\begin{equation*}
d_{S}\left( x\right) \leq \dfrac{\lbrack g\left( x\right) -b]_{+}}{d_{\ast
}\left( 0_{n},\partial g\left( P_{S}\left( x\right) \right) \right) },
\end{equation*}%
where $P_{S}\left( x\right) $ is the metric projection set of $x$ on $S$
with respect to the norm $\left\Vert \cdot \right\Vert .$
\end{proposition}

\begin{proof}
$\left( i\right) $ Inequality $g\left( x\right) \leq b$ turns out to be
equivalent (same solution set, $S$) to its standard linearization by means
of the Fenchel conjugate, $g^{\ast },$ (see, e.g., \cite[Formula (3)]{BCLP21}%
), namely system 
\begin{equation*}
\left\{ \left\langle u,x\right\rangle \leq g^{\ast }\left( u\right) +b,~u\in
\partial g\left( \mathbb{R}^{n}\right) \right\} .
\end{equation*}%
The distance $d_{S}\left( x\right) $ may be computed by means of \cite[Lemma
1]{CGP08}, yielding (with the convention $\frac{0}{0}:=0$)%
\begin{eqnarray*}
d_{S}\left( x\right) &=&\sup \left\{ \left. \frac{\left[ \left\langle
v,x\right\rangle -\alpha \right] _{+}}{\left\Vert v\right\Vert _{\ast }}%
\right\vert \left( v,\alpha \right) \in \mathrm{conv}\left\{ \left(
u,g^{\ast }\left( u\right) +b\right) ,~u\in \partial g\left( \mathbb{R}%
^{n}\right) \right\} \right\} \\
&\geq &\sup \left\{ \left. \frac{\left[ \left\langle u,x\right\rangle
-\left( g^{\ast }\left( u\right) +b\right) \right] _{+}}{\left\Vert
u\right\Vert _{\ast }}\right\vert u\in \partial g\left( \mathbb{R}%
^{n}\right) \right\} \\
&\geq &\sup \left\{ \left. \frac{\left[ g\left( x\right) -b\right] _{+}}{%
\left\Vert u\right\Vert _{\ast }}\right\vert u\in \partial g\left( x\right)
\right\} \\
&=&\frac{\left[ g\left( x\right) -b\right] _{+}}{\inf \left\{ \left\Vert
u\right\Vert _{\ast }\mid u\in \partial g\left( x\right) \right\} }=\dfrac{%
[g\left( x\right) -b]_{+}}{d_{\ast }\left( 0_{n},\partial g\left( x\right)
\right) },
\end{eqnarray*}%
where in the third step we have appealed to the fact that 
\begin{equation*}
g\left( x\right) =g^{\ast \ast }\left( x\right) =\left\langle
u,x\right\rangle -g^{\ast }\left( u\right) \Leftrightarrow u\in \partial
g\left( x\right) .
\end{equation*}

$\left( ii\right) $ It follows from \cite[Lemma 2$\left( ii\right) $]{CGP10}%
. Observe that for $x\in S\ $we apply the convention $\frac{0}{0}:=0,$
whereas for $x\notin S$ the existence of a Slater point entails that $%
P_{S}\left( x\right) $ is not a minimizer of $g$ (since $g\left( P_{S}\left(
x\right) \right) =0$), and then $d_{\ast }\left( 0_{n},\partial g\left(
P_{S}\left( x\right) \right) \right) >0.$
\end{proof}

\begin{remark}
\emph{In many cases it is not difficult to see that }%
\begin{equation*}
b\mapsto \delta \left( b\right) :=d_{\ast }\left( 0_{n},\partial g\left(
g^{-1}\left( b\right) \right) \right) 
\end{equation*}%
\emph{is a positive nondecreasing function on the interval }$\left] \inf_{%
\mathbb{R}^{n}}g,+\infty \right[ $\emph{\ (we are assuming the nontrivial
case when }$g$\emph{\ is not constant, hence not bounded above). Here }$%
\inf_{\mathbb{R}^{n}}g$\emph{\ could be }$-\infty $\emph{\ and }$\partial
g\left( g^{-1}\left( b\right) \right) =\bigcup_{g(y)=b}\partial g\left(
y\right) .$\emph{\ For instance, if }$g\left( x_{1},x_{2}\right)
=e^{x_{1}}+e^{x_{2}},$\emph{\ with the Euclidean norm in }$\mathbb{R}^{2},$ 
\emph{then }$\delta \left( b\right) =b/\sqrt{2}$\emph{\ for }$b>0.$\emph{\
Accordingly, item }$\left( ii\right) $ \emph{in the previous lemma entails }$%
d_{S}\left( x\right) \leq \lbrack g\left( x\right) -b]_{+}/\delta \left(
b\right) .$
\end{remark}

\begin{corollary}
\label{Cor_distance}Let $\overline{b}\in \mathbb{R}^{m}\setminus \Theta _{c}$
and assume that $S_{i}:=\left\{ x\in \mathbb{R}^{n}\mid g_{i}(x)\leq 
\overline{b}_{i}\right\} \neq \emptyset ,$ $i=1,...,m.$ Then, the following
statements hold:

$\left( i\right) $ Let $\emptyset \neq C\subset \mathbb{R}^{n}$ be a closed
convex set such that,\ for each $i\in \{1,...,m\}$, there exists an upper
bound $u_{i}\geq d_{\ast }\left( 0_{n},\partial g_{i}\left( x\right) \right) 
$ for all $x\in C.$ Then,%
\begin{equation}
d_{p}\left( \overline{b},\Theta _{c}\right) ^{p}\leq \inf_{x\in
C}\sum\limits_{i=1}^{m}\left( u_{i}\right) ^{p}d_{S_{i}}^{p}\left( x\right)
=\inf_{x\in \mathbb{R}^{n}}\sum\limits_{i=1}^{m}\left( u_{i}\right)
^{p}d_{S_{i}}^{p}\left( x\right) +I_{C}\left( x\right) ,  \label{eq_upper}
\end{equation}%
where $I_{C}$ is the indicator function of $C$; i.e., $I_{C}\left( x\right)
=0$ if $x\in C$ and $I_{C}\left( x\right) =+\infty $ if $x\in \mathbb{R}%
^{n}\setminus C.$

$\left( ii\right) $ Assume that for each $i\in \{1,...,m\}$ there exists a
lower bound $0<l_{i}\leq d_{\ast }\left( 0_{n},\partial g_{i}\left(
P_{S_{i}}\left( x\right) \right) \right) $ for all $x\in \mathbb{R}%
^{n}\setminus S_{i}$. Then, 
\begin{equation}
d_{p}\left( \overline{b},\Theta _{c}\right) ^{p}\geq \inf_{x\in \mathbb{R}%
^{n}}\sum\limits_{i=1}^{m}\left( l_{i}\right) ^{p}d_{S_{i}}^{p}\left(
x\right) .  \label{eq_lower}
\end{equation}
\end{corollary}

\begin{proof}
$\left( i\right) $ comes straightforwardly from Propositions \ref%
{Prop_distance_slack} and \ref{Prop generalized Ascoli} $\left( i\right) ,$
taking into account the obvious fact that $\inf_{x\in \mathbb{R}%
^{n}}\sum\limits_{i=1}^{m}[g_{i}\left( x\right) -\overline{b}%
_{i}]_{+}^{p}\leq \inf_{x\in C}\sum\limits_{i=1}^{m}[g_{i}\left( x\right) -%
\overline{b}_{i}]_{+}^{p}.$

$\left( ii\right) $ follows immediately from Propositions \ref%
{Prop_distance_slack} and \ref{Prop generalized Ascoli} $\left( ii\right) .$
\end{proof}

Provided that $C$, $u=\left( u_{i}\right) _{i=1,...,m},$ $l=\left(
l_{i}\right) _{i=1,...,m}$ satisfy the conditions of the previous corollary,
we consider the argmin sets coming from (\ref{eq_upper}) and (\ref{eq_lower}%
):%
\begin{align*}
\mathcal{A}\left( C,u\right) & :=\arg \min \sum\limits_{i=1}^{m}\left(
u_{i}\right) ^{p}d_{S_{i}}^{p}\left( x\right) +I_{C}\left( x\right) ,\bigskip
\\
\mathcal{A}\left( l\right) & :=\arg \min \sum\limits_{i=1}^{m}\left(
l_{i}\right) ^{p}d_{S_{i}}^{p}\left( x\right) .
\end{align*}%
Then we can state another corollary of Proposition \ref{Prop generalized
Ascoli}, appealing also to Corollary \ref{Corollary_diff}. Indeed, it brings
to light the advantages of appealing to $\mathcal{A}\left( C,u\right) $ and $%
\mathcal{A}\left( l\right) ,$ instead of working directly with $\arg \min
\sum\limits_{i=1}^{m}[g_{i}\left( x\right) -\overline{b}_{i}]_{+}^{p}.\,$The
key point is that, in the current case in which $p\geq 2,$\textbf{\ }each
function $d_{S_{i}}^{p}$ is differentiable in $\mathbb{R}^{n}$ (see Lemma %
\ref{Lem_dist_K}\textbf{))}, which allows to appeal to Corollary \ref%
{Corollary_diff}, while this is not the case of $[g_{i}\left( \cdot \right) -%
\overline{b}_{i}]_{+}^{p}.$

Hereafter in this section we consider that $\mathbb{R}^{n}$ is endowed with
the Euclidean norm $\left\Vert \cdot \right\Vert _{2}$ and $P_{S}\left(
x\right) $ will denote the unique projection point of $x.$on a closed convex
set $S.$

\begin{corollary}
\label{cor_Argmin}Keeping the previous notation, assume that $\mathcal{A}%
\left( C,u\right) $ and $\mathcal{A}\left( l\right) $ are nonempty. Then, we
have that:

$\left( i\right) $ $d_{S_{i}}$ is constant on both $\mathcal{A}\left(
C,u\right) $ and $\mathcal{A}\left( l\right) ,$ for each $i=1,...,m$;

$\left( ii\right) $ For each $i=1,...,m,$ let us denote by $d_{i}^{+}$ and $%
d_{i}^{-}$ the constant values of $u_{i}d_{S_{i}}\left( \cdot \right) $ and $%
l_{i}d_{S_{i}}\left( \cdot \right) $ on $\mathcal{A}\left( C,u\right) $ and $%
\mathcal{A}\left( l\right) ,$ respectively, and let $d^{+}=\left(
d_{i}^{+}\right) _{i=1,...,m}$ and $d^{-}=\left( d_{i}^{-}\right)
_{i=1,...,m}$. Then,%
\begin{equation*}
d_{p}\left( \overline{b},\Theta _{c}\right) \leq \left\Vert d^{+}\right\Vert
_{p}.
\end{equation*}%
If, in addition, for each $i=1,...,m$ there exists $\widehat{x}_{i}\in 
\mathbb{R}^{n}$ such that $g_{i}(\widehat{x}_{i})<\overline{b}_{i},$ then%
\begin{equation*}
d_{p}\left( \overline{b},\Theta _{c}\right) \geq \left\Vert d^{-}\right\Vert
_{p}.
\end{equation*}
\end{corollary}

\begin{proof}
$\left( i\right) $ Regarding $\mathcal{A}\left( l\right) ,$ the statement
coincides with the one of Theorem \ref{Theorem_111} (in the case when\textbf{%
\ }$p\geq 2)$ just replacing each $\alpha _{i}$ with $\left( l_{i}\right)
^{p}.\,$With respect to $\mathcal{A}\left( C,u\right) ,$ the statement comes
from an analogous argument to the one of that theorem, just by adding the
nondifferentiable mapping $I_{C}.$ For completeness, we include here a
sketch of the proof. Observe that all functions $x\mapsto \left(
u_{i}\right) ^{p}d_{S_{i}}^{p}\left( x\right) $ are convex and
differentiable in $\mathbb{R}^{n},$ and $x\mapsto I_{C}\left( x\right) $ is
a proper lower semicontinuous convex function from $\mathbb{R}^{n}$ to $%
\left] -\infty ,+\infty \right] .$ Hence, the regularity condition (\ref%
{eq_regul_cond}) is satisfied, yielding%
\begin{equation*}
\mathcal{A}\left( C,u\right) =\left\{ x\in \mathbb{R}^{n}\mid 0_{n}\in
\sum\limits_{i=1}^{m}\left( u_{i}\right) ^{p}\nabla d_{S_{i}}^{p}\left(
x\right) +\partial I_{C}\left( x\right) \right\} \left( \neq \emptyset
\right) .
\end{equation*}

From Corollary \ref{Corollary_diff}, for each $i=1,...,m,$ we have that $%
\nabla d_{S_{i}}^{p}$ is constant on $\mathcal{A}\left( C,u\right) ,$ hence $%
d_{S_{i}}$ is also constant on\textbf{\ }$\mathcal{A}\left( C,u\right) $
since taking norms we have 
\begin{equation*}
\left\Vert \nabla d_{S_{i}}^{p}\left( x\right) \right\Vert =\left\Vert
pd_{S_{i}}^{p-2}\left( x\right) \left( x-P_{i}\left( x\right) \right)
\right\Vert =pd_{S_{i}}^{p-1}\left( x\right) ,\text{ for each }x\in \mathcal{%
A}\left( C,u\right) ,
\end{equation*}%
where $P_{i}\left( x\right) $ denotes the projection of $x$ on $S_{i}$
(recall again Lemma \ref{Lem_dist_K}).\textbf{\ }

$\left( ii\right) $ follows immediately from $\left( i\right) $ and
Corollary \ref{Cor_distance}.
\end{proof}

\subsection{Linear systems}

This subsection is devoted to the linear case, i.e., where $%
g_{i}(x)=\left\langle a_{i},x\right\rangle ,$ for some $a_{i}\in \mathbb{R}%
^{n},$ $i=1,...,m$. In this particular case, obviously $\partial g_{i}\left(
x\right) =\left\{ a_{i}\right\} $ for all $x\in \mathbb{R}^{n},$ $i=1,...,m.$
Let us consider $\overline{b}$ such that 
\begin{equation}
\sigma \left( \overline{b}\right) =\left\{ \left\langle a_{i},x\right\rangle
\leq \overline{b}_{i},\text{ }i=1,...,m\right\}  \label{eq_linear_system}
\end{equation}%
is inconsistent and for each $i$ there exists $\widehat{x}_{i}\in \mathbb{R}%
^{n}$ such that $\left\langle a_{i},\widehat{x}_{i}\right\rangle <\overline{b%
}_{i}$ (observe that it is always held when $a_{i}\neq 0_{n}$ or $\overline{b%
}_{i}>0).$ According to the notation of Corollary \ref{Cor_distance}, we can
choose:%
\begin{equation*}
C=\mathbb{R}^{n},\,l_{i}=u_{i}=\left\Vert a_{i}\right\Vert _{\ast
},i=1,...,m.
\end{equation*}%
Hence $\mathcal{A}\left( C,u\right) =\mathcal{A}\left( l\right) ,$ and $%
d_{i}^{+}=d_{i}^{-}$ for all $i.$ Let us denote by $\mathcal{A}:\mathcal{=A}%
\left( C,u\right) $ and $\overline{d}:=\left( d_{i}^{+}\right) _{i=1,...,m}.$

The following corollary follows straightforwardly from Corollary \ref%
{cor_Argmin}.

\begin{corollary}
\label{Cor_argmin_slack}Under the current assumptions, we have%
\begin{equation*}
d_{p}\left( \overline{b},\Theta _{c}\right) =\left\Vert \overline{d}%
\right\Vert _{p},
\end{equation*}%
where $d_{i}^{+}=\left\Vert a_{i}\right\Vert _{\ast }d_{S_{i}}\left(
x\right) =[\left\langle a_{i},x\right\rangle -\overline{b}_{i}]_{+},$ for
all $x\in \mathcal{A}.$ Moreover $\sigma \left( \overline{b}+\overline{d}%
\right) $ is a consistent system nearest to $\sigma \left( \overline{b}%
\right) .$
\end{corollary}

The next result is devoted to provide an operative expression for
determining $\overline{d}$ with the Euclidean norm in both the space of
variables and the space of parameters$.$ For simplicity all norms are
denoted by $\left\Vert \cdot \right\Vert ,$ $A$ represents the matrix whose
rows are $a_{i}^{\prime },$ $i=1,\ldots ,m,$ $A^{\prime }$ denotes its
transpose and, for any $y\in \mathbb{R}^{m},$ $\left[ y\right] _{+}$ denotes
positive part coordinate by coordinate; i.e.,%
\begin{equation*}
\left[ y\right] _{+}:=\left( \left[ y_{i}\right] _{+}\right) _{i=1,\ldots
,m}.
\end{equation*}

\begin{theorem}
\label{thm conditions} The following conditions are equivalent:

$\left( i\right) $ $\left( x^{0},h^{0}\right) \in \mathcal{A\times }\left\{ 
\overline{d}\right\} ;$

$\left( ii\right) $ $\left( x^{0},h^{0}\right) $ is a solution of the
system, in the variable $\left( x,h\right) ,$%
\begin{equation}
\left\{ 
\begin{array}{c}
\left[ Ax-\overline{b}\right] _{+}=h, \\ 
A^{\prime }h=0_{n}.%
\end{array}%
\right.  \label{eq_13}
\end{equation}

$\left( iii\right) $ $\left( x^{0},h^{0}\right) $ is an optimal solution of
the quadratic problem, in the variable $\left( x,h\right) ,$%
\begin{equation}
\begin{array}{ll}
\min & \left\langle h,h\right\rangle \\ 
s.t. & Ax\leq \overline{b}+h, \\ 
& h\geq 0_{m}.%
\end{array}
\label{eq_15}
\end{equation}
\end{theorem}

\begin{proof}
Let us see $\left( i\right) \Rightarrow \left( ii\right) .$ Let $\left(
x^{0},h^{0}\right) \in \mathcal{A\times }\left\{ \overline{d}\right\} ,$
i.e., $x^{0}\in \mathcal{A}$ and $h^{0}=\overline{d}.$ By Corollary\textbf{\ 
}\ref{Cor_argmin_slack}, $h_{i}^{0}(=d_{i}^{+})=\left[ a_{i}^{\prime }x^{0}-%
\overline{b}_{i}\right] _{+},$ for all $i.$ Moreover, the optimality
condition 
\begin{equation*}
x^{0}\in \mathcal{A}:\mathcal{=}\arg \min \sum\limits_{i=1}^{m}\left\Vert
a_{i}\right\Vert ^{2}d_{S_{i}}^{2}\left( x\right)
\end{equation*}%
is equivalent to 
\begin{equation}
0_{n}=\sum\limits_{i=1}^{m}\left\Vert a_{i}\right\Vert ^{2}\nabla
d_{S_{i}}^{2}\left( x^{0}\right) =2\sum\limits_{i=1}^{m}\left\Vert
a_{i}\right\Vert d_{S_{i}}\left( x^{0}\right) a_{i}=2\sum\limits_{i=1}^{m} 
\left[ a_{i}^{\prime }x^{0}-\overline{b}_{i}\right] _{+}a_{i};
\label{eq_001}
\end{equation}%
in other words%
\begin{equation*}
0_{n}=\sum\limits_{i=1}^{m}h_{i}^{0}a_{i}=A^{\prime }h^{0}.
\end{equation*}

So, $\left( x^{0},h^{0}\right) $ is a solution of system (\ref{eq_13}).

$\left( ii\right) \Rightarrow \left( i\right) $ Let $\left(
x^{0},h^{0}\right) $ be a solution of (\ref{eq_13}); i.e., $h^{0}=\left[
Ax^{0}-\overline{b}\right] _{+}$ and%
\begin{equation*}
0_{n}=\sum\limits_{i=1}^{m}h_{i}^{0}a_{i}=A^{\prime }h^{0}.
\end{equation*}%
Then, by repeating the previous argument of (\ref{eq_001}), we have 
\begin{equation*}
0_{n}=\sum\limits_{i=1}^{m}\left\Vert a_{i}\right\Vert ^{2}\nabla
d_{S_{i}}^{2}\left( x^{0}\right) ,
\end{equation*}%
which means that $x^{0}\in \mathcal{A}.$ Then, appealing again to Corollary %
\ref{Cor_argmin_slack}, we deduce $h^{0}=\overline{d}.$

Now, let us prove $\left( ii\right) \Leftrightarrow \left( iii\right) .$ By
the Karush-Kuhn-Tucker (KKT, in brief) conditions, $\left(
x^{0},h^{0}\right) $ is an optimal solution of (\ref{eq_15}) if and only if
there exist $\lambda ,\mu \in \mathbb{R}_{+}^{m}$ such that 
\begin{equation}
\left\{ 
\begin{array}{c}
-\dbinom{0_{n}}{2h^{0}}=\dbinom{A^{\prime }}{-I_{m}}\lambda +\dbinom{%
0_{n\times m}}{-I_{m}}\mu , \\ 
\left( Ax^{0}-\overline{b}-h^{0}\right) ^{\prime }\lambda =0,\text{ }-\left(
h^{0}\right) ^{\prime }\mu =0, \\ 
Ax^{0}-\overline{b}-h^{0}\leq 0_{m},~h^{0}\geq 0_{m}.%
\end{array}%
\right.  \label{eq_17}
\end{equation}%
So, $A^{\prime }\lambda =0_{n},$ and $h^{0}=\frac{\lambda +\mu }{2}.$
Moreover, $h_{i}^{0}\mu _{i}=0$ for all $i.$ Let us see that $\mu =0_{m}.$
If $h_{i}^{0}=0,$ then $\lambda _{i}+\mu _{i}=0,$ which entails $\lambda
_{i}=\mu _{i}=0,$ while, if $h_{i}^{0}>0,$ then $\mu _{i}=0.$ Therefore 
\begin{equation}
h^{0}=\frac{\lambda }{2}  \label{eq_16}
\end{equation}%
and, so, 
\begin{equation*}
A^{\prime }h^{0}=0_{n}.
\end{equation*}%
Let us see that $\left[ Ax^{0}-\overline{b}\right] _{+}=h^{0}.$ Observe that 
$\left( a_{i}^{\prime }x^{0}-\overline{b}_{i}-h_{i}^{0}\right) \lambda
_{i}=0 $ for all $i.$ If $a_{i}^{\prime }x^{0}-\overline{b}_{i}<0,$ then $%
a_{i}^{\prime }x^{0}-\overline{b}_{i}-h_{i}^{0}<0,$ thus we have $\lambda
_{i}=0$ and 
\begin{equation*}
h_{i}^{0}=\frac{\lambda _{i}}{2}=0.
\end{equation*}%
If $a_{i}^{\prime }x^{0}-\overline{b}_{i}>0,$ then $h_{i}^{0}>0$ and $%
\lambda _{i}>0,$ yielding $a_{i}^{\prime }x^{0}-\overline{b}%
_{i}-h_{i}^{0}=0. $ Finally, if $a_{i}^{\prime }x^{0}-\overline{b}_{i}=0$,
then $h_{i}^{0}\lambda _{i}=0$, and from (\ref{eq_16}) we have $h_{i}^{0}=0.$
So, 
\begin{equation*}
\left[ a_{i}^{\prime }x^{0}-\overline{b}_{i}\right] _{+}=h_{i}^{0},\text{
for all }i,
\end{equation*}%
and consequently $\left( x^{0},h^{0}\right) $ is a solution of (\ref{eq_13}).

Reciprocally, if $\left( x^{0},h^{0}\right) $ is a solution of (\ref{eq_13})
and we consider 
\begin{equation*}
\lambda =2h^{0}\text{ and }\mu =0_{m},
\end{equation*}%
it can be easily seen that $x^{0},h^{0},\lambda $ and $\mu $ satisfy the KKT
conditions (\ref{eq_17}), and then $\left( x^{0},h^{0}\right) $ is an
optimal solution for problem (\ref{eq_15}).
\end{proof}

\end{document}